\documentclass[12pt,leqno,twoside]{amsart}
\usepackage{amssymb,amsmath,amsthm,soul,color}
\usepackage{t1enc}
\usepackage[cp1250]{inputenc}
\usepackage{a4,indentfirst,latexsym}
\usepackage{graphics}
\usepackage{mathrsfs}
\usepackage{cite,enumitem,graphicx}
\usepackage[colorlinks=true,urlcolor=blue,
citecolor=red,linkcolor=blue,linktocpage,pdfpagelabels,
bookmarksnumbered,bookmarksopen]{hyperref}
\usepackage[english]{babel}
\usepackage[left=2.61cm,right=2.61cm,top=2.72cm,bottom=2.72cm]{geometry}
\usepackage[metapost]{mfpic}
\usepackage[colorinlistoftodos]{todonotes}
\usepackage[normalem]{ulem}
\newtheorem{Th}{Theorem}[section]
\newtheorem{Prop}[Th]{Proposition}
\newtheorem{Lem}[Th]{Lemma}
\newtheorem{Cor}[Th]{Corollary}

\newenvironment{altproof}[1]
{\noindent%\addvspace{0.3cm}
{\em Proof of {#1}}.}
{\nopagebreak\mbox{}\hfill $\Box$\par\addvspace{0.5cm}}

\newcommand{\wt}{\widetilde}

   \newcommand{\vp}{\varphi}
   
   \newcommand{\eps}{\varepsilon}

   \def\div{\mathop{\mathrm{div}}}

   \def\Z{\mathbb{Z}}

   \def\R{\mathbb{R}}

   \def\curl{\mathrm{curl}}
   \def\dim{\mathrm{dim}}

   \def\V{\mathcal{V}}
   \def\E{\mathcal{E}}
   \def\J{\mathcal{J}}

   \def\W{\mathcal{W}}
   \def\D{\mathcal{D}}

     \def\M{\mathcal{M}}
     \def\tJ{\wt{\cJ}}

\def\rh{\rightharpoonup}

\newcommand{\cC}{{\mathcal C}}
\newcommand{\cD}{{\mathcal D}}
\newcommand{\cE}{{\mathcal E}}

\newcommand{\cG}{{\mathcal G}}
\newcommand{\cH}{{\mathcal H}}
\newcommand{\cI}{{\mathcal I}}
\newcommand{\cJ}{{\mathcal J}}
\newcommand{\cK}{{\mathcal K}}
\newcommand{\cL}{{\mathcal L}}
\newcommand{\cM}{{\mathcal M}}
\newcommand{\cN}{{\mathcal N}}

\newcommand{\cT}{{\mathcal T}}

\newcommand{\cV}{{\mathcal V}}
\newcommand{\cW}{{\mathcal W}}

\renewcommand{\dim}{{\rm dim}\,}

\newcommand{\al}{\alpha}

\newcommand{\ga}{\gamma}

\newcommand{\Ga}{\Gamma}
\newcommand{\Om}{\Omega}

\def\curlop{\nabla\times}
\newcommand{\weakto}{\rightharpoonup}
\newcommand{\pa}{\partial}

\newcommand{\tX}{\widetilde{X}}
\newcommand{\tu}{\widetilde{u}}
\newcommand{\tv}{\widetilde{v}}

\newcommand{\cTto}{\stackrel{\cT}{\longrightarrow}}

\numberwithin{equation}{section}

\title[Nonlinear Curl-curl problem]{Multiple solutions to a nonlinear curl-curl problem\\ in $\mathbb{R}^3$}

\author[J. Mederski]{Jaros\l aw Mederski}
\author[J. Schino]{Jacopo Schino}
\author[A. Szulkin]{Andrzej Szulkin}

\address[J. Mederski]{\newline\indent
	{\newline\indent 
		Institute of Mathematics,
		\newline\indent 
		Polish Academy of Sciences,
		\newline\indent 
		ul. \'Sniadeckich 8, 00-656 Warsaw, Poland
		\newline\indent 
		and
		\newline\indent 
		Faculty of Mathematics and Computer Science
		\newline\indent 
		Nicolaus Copernicus University
		\newline\indent
		ul. Chopina 12/18, 87-100 Toru\'n, Poland}
}
\email{\href{mailto:jmederski@impan.pl}{jmederski@impan.pl}}
\address[J. Schino]{\newline\indent
	{\newline\indent 
		Institute of Mathematics,
		\newline\indent 
		Polish Academy of Sciences,
		\newline\indent 
		ul. \'Sniadeckich 8, 00-656 Warsaw, Poland}
}
\email{\href{mailto:jschino@poliba.it}{jschino@impan.pl}}

\address[A. Szulkin]{\newline\indent 
	Department of Mathematics,
	\newline\indent 
	Stockholm University,
	\newline\indent 
	106 91 Stockholm, Sweden
	}
\email{\href{mailto:andrzejs@math.su.se}{andrzejs@math.su.se}}


\begin{document}
\begin{abstract} We look for ground states and bound states  $E:\R^3\to\R^3$ to the  curl-curl problem
$$\nabla\times(\nabla\times E)= f(x,E) \qquad\textnormal{ in } \mathbb{R}^3$$
which originates from nonlinear Maxwell equations.
The energy functional associated with this problem is strongly indefinite due to the infinite dimensional kernel of $\nabla\times(\nabla\times \cdot)$. The growth of the nonlinearity $f$ is controlled by an $N$-function $\Phi:\R\to [0,\infty)$ such that $\displaystyle\lim_{s\to 0}\Phi(s)/s^6=\lim_{s\to+\infty}\Phi(s)/s^6=0$. We prove the existence of a ground state, i.e. a least energy nontrivial solution, and the existence of infinitely many geometrically distinct bound states. We improve previous results concerning ground states of curl-curl problems. Multiplicity results for our problem have not been studied so far in $\R^3$ and in order to do this we construct a suitable critical point theory. It is applicable to a wide class of strongly indefinite problems, including this one and Schr\"odinger equations.
\end{abstract}

\maketitle

\section*{Introduction}
\setcounter{section}{1}

We look for weak solutions to the semilinear curl-curl problem
\begin{equation}\label{eq}
\nabla\times(\nabla\times E)= f(x,E), \qquad x\in\R^3,
\end{equation}
originating from the Maxwell equations where $E(x)\cos(\omega t)$ is a time-harmonic electric field in a nonlinear medium and $f(x,E)$ models a nonlinear polarization in the medium, see \cite{Mederski2015,Stuart91,Stuart93} and the references therein. Another motivation has been provided by Benci and Fortunato \cite{BenFor} who introduced a model for a unified field theory for classical electrodynamics based on a semilinear perturbation of the Maxwell equations in the spirit of the Born-Infeld theory \cite{BornInfeld}. In the magnetostatic case in which the electric field vanishes and the magnetic field is independent of time, this  leads to an equation of the form \eqref{eq} with $E$ replaced by $A$, the gauge potential related to the magnetic field.\\
\indent The semilinear curl-curl problem in $\R^3$ has been solved for the first time in \cite{BenForAzzAprile} in the cylindrically symmetric setting. If $f(x,E)$ depends only on $|E|$, then one can restrict the considerations to
the fields of the form
\begin{equation}\label{eq:sym1}
E(x)=\al(r,x_3)\begin{pmatrix}-x_2\\x_1\\0\end{pmatrix},\qquad r=\sqrt{x_1^2+x_2^2}
\end{equation}
which are divergence-free, so $\curlop(\curlop E)=-\Delta E$ and one can study \eqref{eq} by means of standard variational methods (however, there may still exist solutions which are not of this form). Other results in the cylindrically symmetric setting have been obtained in \cite{DAprileSiciliano,BDPR:2016,HirschReichel,Zeng,Mandel}. We would also like to mention that travelling waves of similar form for a system of nonlinear Maxwell equations have been studied by Stuart and Zhou in \cite{Stuart91,Stuart93,StuartZhou05, StuartZhou10} for asymptotically linear $f$ and  by McLeod,  Stuart and Troy \cite{McLeodStuartTroy} for a cubic nonlinearity. This approach requires again cylindrically symmetric media and involves ODE methods  which are not applicable if $f$ in \eqref{eq} lacks this symmetry.\\
\indent In the  media which are not cylindrically symmetric the problem is much more challenging since the {\em curl-curl operator} $\curlop(\curlop\cdot)$ has an infinite-dimensional kernel 
 consisting of all gradient vector fields. Hence the energy functional associated with \eqref{eq} 
\begin{equation}\label{eq:action}
\E(E)=\frac12\int_{\R^3}|\curlop E|^2\,dx - \int_{\R^3} F(x,E)\,dx
\end{equation}
where $f=\partial_E F$
is unbounded from above and from below and its critical points may have infinite Morse index. For instance, this is the case in a model example
\begin{equation}\label{eq:modelEx}
f(x,E)=\Gamma(x)\min\{|E|^{p-2},|E|^{q-2}\}E\quad\hbox{ with }2<p<6<q
\end{equation}
where $\Gamma\in L^{\infty}(\R^3)$ is $\Z^3$-periodic, positive and bounded away from $0$. Let $\cD(\curl,\Phi)$ be the space of functions  $E$ such that $\nabla\times E$ is square integrable and $E$ is in the Orlicz space $L^\Phi(\R^3,\R^3)$ for an appropriate growth function $\Phi$; see the next section for a more accurate definition. Then $\mathcal{E}\in \cC^1(\cD(\curl,\Phi), \R)$ and  critical points of $\cE$ are weak solutions to \eqref{eq}. 
In addition to these problems related to the strongly indefinite geometry of $\E$, we also have to deal with issues related to the lack of compactness. Namely, the functional $\E'$ is not (sequentially) {\em weak-to-weak$^*$ continuous}, i.e.  weak convergence $E_n\weakto E$ in $\D(\curl,\Phi)$ does not imply that $\E'(E_n)\weakto \E'(E)$ in $\D(\curl,\Phi)^*$, hence we do not know whether the weak limit of a bounded Palais-Smale sequence is a critical point.\\
\indent Similar difficulties have already appeared in curl-curl problems on bounded domains in Bartsch and Mederski \cite{BartschMederski} where a generalized Nehari manifold approach inspired by Szulkin and Weth \cite{SzulkinWeth} has been developed to overcome strong indefiniteness. Other approaches have been developed in subsequent work \cite{BartschMederskiJFA,MederskiJFA2018}; see also the survey \cite{BartschMederskiJFTA}. Note that on a bounded domain there is no problem with lack of weak-to-weak$^*$ continuity of $\cE'$ since a variant of the Palais-Smale condition is satisfied under some constraints. In $\R^3$ however, one has to make a careful concentration-compactness analysis on a suitable generalized Nehari manifold $\cN_\E$; this has been demonstrated in \cite{Mederski2015} which seems to be the only work on ground states of \eqref{eq} in the nonsymmetric setting.\\ 
\indent In the present work we consider a larger class of nonlinearities which have supercritical growth at $0$ and subcritical growth at infinity; this is in the spirit of the zero mass case of Berestycki and Lions \cite{BerLions}, see condition (N2) below. However, as shown by the examples below, we admit nonlinearities which are more general than in \eqref{eq:modelEx}, and this requires a new functional setting for \eqref{eq} as well as a new critical point theory. The reason is that the methods based on the constraint $\cN_{\E}$ (see \eqref{DefOfNehari1} for the definition) cannot be applied straightforwardly here since $\cN_{\E}$ may not be homeomorphic to the unit sphere in the subspace of divergence-free vector fields as in \cite{BartschMederski,Mederski2015}. Our critical point theory for strongly indefinite functionals in Section \ref{sec:criticaslpoitth} also solves the problem of multiplicity of bound states. It has not been considered so far, not even for \eqref{eq:modelEx}. Note that although $\cE$ has the classical linking geometry, the well-known linking results, e.g. of Benci and Rabinowitz \cite{BenciRabinowitz}, are not applicable due to the lack of weak-to-weak$^*$ continuity of $\cE'$.\\
\indent In order to state our main result we assume that
the growth of  $f$ is controlled by a strictly convex $N$-function $\Phi:\R\to [0,\infty)$ of class $\cC^1$ such that
\begin{itemize}
\item[(N1)] $\Phi$ satisfies the $\Delta_2$- and the $\nabla_2$-condition globally.
\item[(N2)] $\displaystyle\lim_{s\to 0}\frac{\Phi(s)}{s^6}=\lim_{s\to\infty}\frac{\Phi(s)}{s^6}=0$.
\item[(N3)] $\displaystyle\lim_{s\to\infty}\frac{\Phi(s)}{s^2}=\infty$.
\end{itemize}
$N$-functions and condition (N1) will be introduced in the next section and are  standard in the theory of Orlicz spaces \cite{RaoRen}. (N2) is inspired by \cite{BerLions} and (N2), (N3) describe supercritical behaviour at $0$ and superquadratic but subcritical at infinity.
We collect our assumptions on the nonli\-nearity $F(x,u)$.
\begin{itemize}
\item[(F1)] $F:\R^3\times\R^3\to\R$ is differentiable with respect to the second variable $u\in\R^3$ for a.e. $x\in\R^3$, and $f=\pa_uF:\R^3\times\R^3\to\R^3$ is a Carath\'eodory function (i.e.\ measurable in $x\in\R^3$, continuous in $u\in\R^3$ for a.e.\ $x\in\R^3$). Moreover, $f$ is $\Z^3$-periodic in $x$, i.e. $f(x,u)=f(x+y,u)$ for all $u\in \R^3$, and almost all $x\in\R^3$ and $y\in\Z^3$.
\item[(F2)] $F$ is uniformly strictly convex with respect to $u\in\R^3$, i.e.\ for any compact $A\subset(\R^3\times\R^3)\setminus\{(u,u):\;u\in\R^3\}$
$$
\inf_{\genfrac{}{}{0pt}{}{x\in\R^3}{(u_1,u_2)\in A}}
 \left(\frac12\big(F(x,u_1)+F(x,u_2)\big)-F\left(x,\frac{u_1+u_2}{2}\right)\right) > 0.
$$
\item[(F3)] There are $c_1$, $c_2>0$ such that
$$|f(x,u)|\le c_1\Phi'(|u|)\text{ and }F(x,u)\ge c_2\Phi(|u|)$$
for every $u\in\R^3$ and a.e. $x\in\R^3$.
\item[(F4)] For every $u\in\R^3$ and a.e. $x\in\R^3$ 
$$\langle f(x,u),u\rangle\geq 2F(x,u).$$
\item[(F5)] If $ \langle f(x,u),v\rangle = \langle f(x,v),u\rangle >0$, then
$\ \displaystyle F(x,u) - F(x,v)
 \le \frac{\langle f(x,u),u\rangle^2-\langle f(x,u),v\rangle^2}{2\langle f(x,u),u\rangle}$.
\end{itemize}
We provide some examples. First we note that if $G = G(x,t): \R^3\times\R\to\R$ is differentiable with respect to $t$, $g:=\partial_tG$ is a Carath\'eodory function, $G(x,0)=0$, $M\in GL(3)$ is an invertible $3\times 3$ matrix and
\begin{equation} \label{increasing}
F(x,u) = G(x,|Mu|) \quad  \text{and} \quad t\mapsto g(x,t)/t \text{ is non-decreasing for } t>0,
\end{equation}
then $F$ satisfies (F4) (cf. \cite{SzulkinWeth}) and it is easy to see that (F5) holds. Note that \eqref{increasing} implies $g(x,0)=0$, so $f$ is continuous also at $u=0$.\\
\indent Suppose  $\Ga\in L^\infty(\R^3)$ is $\Z^3$-periodic, positive and bounded away from $0$. Take 
$$F(x,u):=\Gamma(x)W(|Mu|^2)$$ where $W$ is a function of class $\cC^1$, $W(0)=W'(0)=0$ and $t\mapsto W'(t)$ is non-decreasing on $(0,+\infty)$. Then we check that (F1), (F2), (F4) and (F5) are satisfied (here $G(x,t)=\Gamma(x)W(t^2)$, so \eqref{increasing} holds). If $W(t^2)=\frac1p\big((1+|t|^q)^{\frac{p}{q}}-1\big)$ or
$W(t^2)=\min\big\{\frac1p|t|^p+\frac1q-\frac1p,\frac1q|t|^q\big\}$ with $2<p<6<q$, then we can take $\Phi(t)=W(t^2)$ and we see that (F3) holds as well. Note that
if $W'(t)$ is constant on some interval $[a,b]\subset (0,+\infty)$, then 
 \begin{equation} \label{equality}
 0<\displaystyle F(x,u) - F(x,v)
 =\frac{\langle f(x,u),u\rangle^2-\langle f(x,u),v\rangle^2}{2\langle f(x,u),u\rangle}
 \end{equation}
 for $a<|v|<|u|<b$ and a stronger variant of (F5), i.e. \cite[(F5)]{Mederski2015}, is no longer satisfied. So we cannot apply  variational techniques relying on minimization on the Nehari-Pankov manifold $\cN_{\cE}$ (defined in \eqref{DefOfNehari1}) as in \cite{SzulkinWeth,Mederski2015}. 
Moreover, our problem requires a new functional setting. Indeed, 
if we consider 
$W(t^2)=\frac12(|t|^2-1)\ln (1+|t|)-\frac14|t|^2+\frac12|t|$ for $|t|\geq 1$, $W(t^2)=\frac{\ln2}q(|t|^q-1)+\frac14$ for $|t|<1$, then 
\begin{equation*}
f(x,u)=\begin{cases}
\Gamma(x)u\ln(1+|u|)  & \text{ if } |u|\geq 1,\\
\Gamma(x)\ln(2)u|u|^{q-2} & \text{ if } |u|<1
\end{cases}
\end{equation*}
and (F1)--(F5) are satisfied; however, $f$ cannot be controlled by any $N$-function associated with $L^p(\R^3,\R^3)+L^q(\R^3,\R^3)$ for $2<p<6<q$ as in \cite{Mederski2015} or in  other zero mass case problems \cite{BenGrisantiMich,ClappMaia}. As our final example we take $F(x,u)=\Gamma(x) \Phi(|u|)$ where $\Phi(0)=0$, 
 \begin{equation*}
\Phi'(t) =\begin{cases}
t^{r-1}  & \text{ if } t\le 1,\\
t & \text{ if } 1\le t\le 2, \\
at^5/\ln t & \text{ if } t\ge2,
\end{cases}
\end{equation*}
$r>6$ and $a=2^{-4}\ln2$. Obviously, $F$ satisfies \eqref{increasing} and hence (F4), (F5), and \eqref{equality} holds for $1<|u|<2$. It is easy to see that (F1)--(F3) and (N1)--(N3) hold (to check (N1) it is convenient to use Lemma \ref{AllProp}). Note that here $\Phi(t)/t^6\to 0$ but $\Phi(t)/t^p\to\infty$ as $t\to\infty$ for any $p<6$. Note also that in the last two examples we can replace $|u|$ by $|Mu|$. 
 \\
 \indent
Our principal aim is to prove the following result.
\begin{Th}\label{ThMain} Assume that (F1)--(F5) hold. Then: \\
	(a) Equation \eqref{eq} has a ground state solution, i.e. there is a critical point $E\in\mathcal{N}_{\cE}$ of $\E$ such that
	$$\E(E)=\inf_{\mathcal{N}_{\cE}}\E>0$$
	where
	\begin{eqnarray}\label{DefOfNehari1}
	\mathcal{N}_{\cE} &:=& \{E\in \D(\curl,\Phi): E\neq 0,\; 
	\E'(E)[E]=0,\\\nonumber
	&&\hbox{ and }\E'(E)[\nabla\vp]=0\,\hbox{ for any }\vp\in \cC_0^{\infty}(\R^3)\}.
	\end{eqnarray}
	(b) If in addition $F$ is even in $u$,  there is an infinite sequence $(E_n)\subset\cN_{\cE}$ of geometrically distinct solutions of  \eqref{eq}, i.e. solutions such that $(\Z^3\ast E_n)\cap (\Z^3\ast E_m)=\emptyset$ for $n\neq m$, where
	$$\Z^3\ast E_n:=\{E_n(\cdot+y): y\in\Z^3\}.$$
\end{Th}

In our approach we establish a critical point theory on the topological manifold
$$\cM_\cE:=\{E\in \cD(\curl,\Phi):\E'(E)[\nabla\vp]=0\,\hbox{ for any }\vp\in \cC_0^{\infty}(\R^3)\}$$
which contains $\cN_\cE$ as a subset,
and we show that $\cE$ has the mountain pass geometry in $\cM_\cE$ and admits a Cerami sequence at the ground state level $\inf_{\cN_{\cE}}\cE>0$; see the abstract setting and the critical point theory in Section \ref{sec:criticaslpoitth}. In order to find a nontrivial critical point being a ground state one needs to analyze Cerami sequences in the spirit of Lions \cite{Lions84}. However, this is not straightforward because the kernel of the curl-curl operator is not locally compactly embedded into any $L^p$ or Orlicz space and $\cE'$ lacks weak-to-weak$^*$ continuity. Therefore it is difficult to treat this problem by a concentration-compactness argument directly in the space $\cD(\curl,\Phi)$. Based on a crucial convergence result obtained in Proposition \ref{prop}, we prove that $\cE'$ is weak-to-weak$^*$ continuous in $\cM_\E$, see Corollary \ref{CorJweaklycont}. This allows us to find a nontrivial weak limit of the Cerami sequence which is a ground state solution as in Theorem \ref{Th:CrticMulti}(a). 
Moreover, a result on the discreteness of Cerami sequences allows us to find infinitely many geometrically distinct solutions. \\
\indent We would also like to mention that our methods allow to consider Schr\"odinger equations in the zero mass case as in \cite{BenGrisantiMich,ClappMaia} and we are able to obtain new results with improved growth conditions; see Section \ref{sec:Schordingerproblem}.

\section{Preliminaries and variational setting}\label{sec:varsetting}

Here and in the sequel $|\cdot|_q$ denotes the $L^q$-norm.\\
\indent Now, following \cite{RaoRen}, we recall some basic definitions and results about $N$-functions and Orlicz spaces.
A function $\Phi\colon\R\to[0,+\infty)$ is called an $N$-function, or a nice Young function if  it is convex, even and satisfies
$$\Phi(t)=0\Leftrightarrow t=0,\quad\lim_{t\to 0}\frac{\Phi(t)}{t}=0,\quad\hbox{and }\lim_{t\to+\infty}\frac{\Phi(t)}{t}=+\infty.$$
Given an $N$-function $\Phi$, we can associate with it another function $\Psi\colon\R\to[0,+\infty)$ defined by
\[
\Psi(t):=\sup\{s|t|-\Phi(s):s\ge 0\}
\]
which is an $N$-function as well. $\Psi$ is called the {\em complementary function} to $\Phi$ while $(\Phi,\Psi)$ is called a {\em complementary pair} of $N$-functions.\\
\indent 
We recall from \cite[Section I.3]{RaoRen} that $\Phi'$ and $\Psi'$ exist a.e., $\Psi'(t)=\inf\{s\ge 0:\Phi'(s)>t\}$ for $t\ge 0$, $\Psi'(t)=-\Psi'(-t)$ for $t<0$ and $\Psi$ can be expressed as
\[
\Psi(t)=\int_0^{|t|}\Psi'(s)\,ds.
\]
We also recall from \cite[Section II.3]{RaoRen} that $\Phi$ satisfies the $\Delta_2$-{\em condition globally} (denoted $\Phi\in\Delta_2$) if there exists $K>1$ such that for every $t\in\R$
\[
\Phi(2t)\le K\Phi(t)
\]
(here $2$ can be replaced by any constant $a>1$) while $\Phi$ satisfies the $\nabla_2${\em -condition globally} (denoted $\Phi\in\nabla_2$) if there exists $K'>1$ such that for every $t\in\R$
\[
\Phi(K't)\ge 2K'\Phi(t).
\]
The set 
$$L^\Phi:=L^\Phi(\R^3,\R^3):=\Big\{E\colon\R^3\to\R^3\text{ measurable and } \int_{\R^3}\Phi(|E|)<\infty\Big\}$$ is a vector space if $\Phi\in\Delta_2$ globally; in this case it is called an Orlicz space. Moreover, the space $L^\Phi$ (whenever it is actually a vector space) becomes a Banach space (cf. \cite[Theorem III.2.3, Theorem III.3.10]{RaoRen}) if endowed with the norm
$$|E|_{\Phi}:=\inf\Big\{k>0:\int_{\R^3}\Phi\Big(\frac{|E|}{k}\Big)\le 1\Big\}.$$
We can define an equivalent norm on $L^\Phi$ by letting
$$|E|_{\Phi,1}:=\sup\Big\{\int_{\R^3}|E|\, |E'|\, dx: \int_{\R^3}\Psi(|E'|)\, dx \leq 1,\; E'\in L^{\Psi}\Big\},$$
see \cite[Proposition III.3.4]{RaoRen} (note that in \cite{RaoRen} these results are formulated for the space $\cL^\Phi$; however, no distinction needs to be made between $\cL^\Phi$ and $L^\Phi$, see the comment following  \cite[Corollary III.3.12]{RaoRen}). 
Finally, if both $\Phi$ and $\Psi$ satisfy the $\Delta_2$-condition globally, then $L^\Phi$ is reflexive and $L^\Psi$ is its dual \cite[Corollary IV.2.9 and Theorem IV.2.10]{RaoRen}. Similarly, for any measurable $\Om\subset\R^3$ one can define $$L^\Phi(\Om):=\Big\{\xi\colon\Om\to\R\text{ measurable and} \int_{\Om}\Phi(|\xi|)<\infty\Big\}$$
and endow it with the norm $|\cdot|_\Phi$ defined as above. \\
\indent In the lemma below we show that $L^\Phi$ and $L^\Phi(\R^3)^3$ can be identified. The result should be known but we could not find any explicit reference.

\begin{Lem} \label{identify}
The norms of  $L^\Phi= L^\Phi(\R^3,\R^3)$ and $L^\Phi(\R^3)^3$ are equivalent.
\end{Lem}

\begin{proof}
In $L^\Phi$ and $L^\Phi(\R^3)$ we use the norm $|\cdot|_\Phi$ defined above and for $E=(E_1,E_2,E_3)\in L^\Phi(\R^3)^3$ we set $|E|_{\Phi,3}:=\max_{i=1,2,3}|E_i|_\Phi$. Since $\Phi$ is increasing on positive numbers, we have
\[
\int_{\R^3}\Phi\left(\frac{|E_i|}k\right)dx \le \int_{\R^3}\Phi\left(\frac{|E|}k\right)dx, \quad k>0,
\]
hence if the second integral is $\le 1$, so is the first one. Taking the infimum over $k>0$ we obtain $|E_i|_\Phi\le|E|_\Phi$ and $|E|_{\Phi,3}\le|E|_\Phi$. On the other hand, since $\Phi$ is convex,
\[
\int_{\R^3}\Phi\left(\frac{|E|}{3k}\right)dx \le \frac13\sum_{i=1}^3 \int_{\R^3}\Phi\left(\frac{|E_i|}k\right)dx,
\]
so $\frac13|E|_\Phi \le \max_{i=1,2,3}|E_i|_\Phi = |E|_{\Phi,3}$.
\end{proof}

\indent Before going on, for the reader's convenience we recall some important facts.
\begin{Lem}\label{AllProp}$\mbox{}$
\begin{itemize}
\item [(i)] The following are equivalent:
	\begin{itemize}
	\item [-] $\Phi\in\Delta_2$ globally;
	\item [-] there exists $K>1$ such that $t\Phi'(t)\le K\Phi(t)$ for every $t\in\R$;
	\item [-] there exists $K'>1$ such that $t\Psi'(t)\ge K'\Psi(t)$ for every $t\in\R$;
	\item [-] $\Psi\in\nabla_2$.
	\end{itemize}
\item [(ii)] For every $E\in L^\Phi$, $E'\in L^\Psi$ there holds
$$\int_{\R^3}|E|\,|E'|\,dx\le\min\{|E|_{\Phi,1}|E'|_{\Psi},|E|_{\Phi}|E'|_{\Psi,1}\}.$$
\item [(iii)] Let $E_n$, $E\in L^\Phi$. Then $|E_n-E|_\Phi\to 0$ implies that $\int_{\R^3}\Phi(|E_n-E|)\,dx\to 0$. If $\Phi\in\Delta_2$ globally, then $\int_{\R^3}\Phi(|E_n-E|)\,dx\to 0$ implies $|E_n-E|_\Phi\to 0$.
\item [(iv)] Let $X\subset L^\Phi$ and suppose $\Phi\in\Delta_2$ globally. Then $X$ is bounded if and only if $\{\int_{\R^3}\Phi(|E|)\,dx:E\in X\}$ is bounded.
\end{itemize}
\end{Lem}

\begin{proof}
(i) follows from \cite[Theorem II.3.3]{RaoRen}; (ii) follows from \cite[Proposition III.3.1 and Formula (III.3.17)]{RaoRen}; (iii) follows from \cite[Theorem III.4.12]{RaoRen}; (iv) follows from \cite[Corollary III.4.15]{RaoRen}.
\end{proof}

From now on we assume (F1)--(F5), (N1)--(N3), $\Phi$ will denote an $N$-function as in (F3) and $\Psi$ will denote its complementary function. Moreover, we will denote by $|\cdot|_\Phi$ any of the two (equivalent) norms defined above, unless differently required.

Let $\D(\curl,\Phi)$ be the completion of $\mathcal{C}_0^{\infty}(\R^3,\R^3)$ with respect to the norm
$$\|E\|_{\curl,\Phi}:=\big(|\curlop E|_2^2+|E|_{\Phi}^2\big)^{1/2}.$$
 The subspace of divergence-free vector fields is defined by
 \[
 \begin{aligned}
 \V
 &:= \left\{v\in \D(\curl,\Phi):\; \int_{\R^3}\langle v,\nabla \vp\rangle\,dx=0
 \text{ for any }\vp\in \cC^\infty_0(\R^3)\right\}\\
 &= \{v\in \D(\curl,\Phi):\; \div v=0\}
 \end{aligned}
 \]
 where $\div v$ is to be understood in the distributional sense. Let $\cD:=\D^{1,2}(\R^3,\R^3)$
 be the completion of $\cC^{\infty}_0(\R^3,\R^3)$ with respect to the norm 
 $$\|u\|_{\D}:=|\nabla u|_2,$$
 and let $\W$ be the closure of
 $\big\{\nabla\vp: \vp\in \cC^{\infty}_0(\R^3)\big\}$ 
 in $L^{\Phi}$.
 \begin{Lem}\label{Emb}
 $L^6(\R^3,\R^3)$ is continuously embedded in $L^{\Phi}$.
 \end{Lem}
\begin{proof}
In view of (N2) it is clear that
$\Phi(t)\leq C|t|^6$ for any $t\in\R$ and some $C>0$. So we can conclude by Lemma \ref{AllProp} (iii).
\end{proof}

The following Helmholtz decomposition holds.
\begin{Lem}\label{defof} $\V$ and $\cW$ are closed subspaces of $\D(\curl,\Phi)$  and
	\begin{equation}\label{HelmholzDec}
	\D(\curl,\Phi)=\V\oplus \W.
	\end{equation}
	Moreover, $\cV\subset\cD$ and the norms $\|\cdot\|_{\cD}$ and $\|\cdot\|_{curl,\Phi}$ are equivalent in $\V$.
\end{Lem}

\begin{proof} Take any $w\in\W$ and a sequence $\vp_n\in\cC_0^{\infty}(\R^3)$ such that $|w -\nabla\vp_n|_{\Phi}\to 0$.
Then for any $\psi\in\cC_0^{\infty}(\R^3,\R^3)$
$$
\int_{\R^3}\langle w,\curlop\psi\rangle\, dx
=\lim_{n\to\infty}\int_{\R^3}\langle\nabla \vp_n,\curlop\psi\rangle\, dx=
\lim_{n\to\infty}\int_{\R^3}\langle\curlop(\nabla \vp_n),\psi\rangle\, dx=0$$
where we have used Lemma \ref{AllProp} (ii) and the fact that $\curlop\psi \in L^\Psi$. Hence $\curlop  w=0$ in the sense of distributions and
$\|w\|_{\curl,\Phi}=|w|_{\Phi}$. Therefore $\W$ is closed in $\D(\curl,\Phi)$; moreover, we easily see that also $\cV$ is closed in $\D(\curl,\Phi)$.\\
\indent Now, take any $E\in \D(\curl,\Phi)$ and $\vp_n\in\cC_0^{\infty}(\R^3,\R^3)$ such that $\vp_n\to E$ in $\D(\curl,\Phi)$. Let $\vp_n^2\in\cC^{\infty}(\R^3)$ be the Newtonian potential of $\div(\vp_n)$, i.e. $\vp_n^2$ solves $\Delta \vp_n^2 = \div(\vp_n)$. Note  that the derivative $\partial_i\vp_n^2$ is the Newtonian potential of $\div(\partial_i\vp_n)$. Since $\vp_n\in\cC_0^{\infty}(\R^3)$, then by \cite[Proposition 1]{Iwaniec}, $\nabla\vp_n^2$ and $ \nabla(\partial_i\vp_n^2) \in L^r(\R^3,\R^3)$ for every $r\in(1,\infty)$. Hence by Lemma \ref{Emb}
$$
\nabla \vp_n^2\in  L^6(\R^3,\R^3)\subset L^{\Phi}
$$ 
and $\vp_n^1:=\vp_n-\nabla \vp_n^2\in L^{\Phi}$. Moreover, $\vp_n^1$ and $\partial_i\vp_n^1\in L^r(\R^3,\R^3)$.
We also have  $\curlop \vp_n^1 = \curlop\vp_n$  and $\div(\vp_n^1)=0$ pointwise. Using these two equalities and integrating by parts gives
 $|\nabla\vp^1_n|_2 = |\curlop \vp^1_n|_2 = |\curlop\vp_n|_2$.  It follows that for $m,n\ge 1$,
$$|\nabla(\vp_n^1-\vp_m^1)|_{2}=|\curlop(\vp_n^1-\vp_m^1)|_{2}=|\curlop(\vp_n-\vp_m)|_{2}\leq \|\vp_n-\vp_m\|_{\curl,\Phi}.$$
Thus $(\vp_n^1)$ is a Cauchy sequence in $\D$. Let $v:=\lim_{n\to\infty}\vp_n^1$ in $\cD$. Then
$$\int_{\R^3}\langle v,\nabla \vp\rangle\,dx=\lim_{n\to\infty}\int_{\R^3}\langle \vp^1_n,\nabla \vp\rangle\,dx=0$$
for any $\vp\in\cC_0^{\infty}(\R^3)$, hence $\div v =0$ and $v\in\cV$.
Moreover,
$$|\curlop(\vp_n^1-v)|_{2}=|\nabla(\vp_n^1-v)|_{2}\to 0,$$
so $\vp_n^1\to v$ in  $\D(\curl,\Phi)$ and
$\nabla \vp^2_n=\vp_n-\vp^1_n\to E-v$ in $\D(\curl,\Phi)$. Since $\cW$ is closed in $\D(\curl,\Phi)$, then $E-v\in \cW$ and we get the decomposition
$$E=v+(E-v)\in \cV+\cW.$$
\indent Now take $v\in \cV\cap\W$. Then $\curlop v =0$, so by \cite[Lemma 1.1(i)]{Leinfelder}, $v=\nabla\xi$ for some $\xi\in W^{1,6}_{loc}(\R^3)$. Since $\div v = 0$, $\xi$ is harmonic and therefore so is $v$. Hence 
\[
0 = -\int_{\R^3} \langle v, \Delta v\rangle\,dx = \int_{\R^3}|\nabla v|^2\,dx
\] 
(integration by parts is allowed because $v\in D^{1,2}(\R^3,\R^3)$). So $v=0$; therefore  $\cV\cap\W=\{0\}$ and we obtain \eqref{HelmholzDec}.
\\
\indent The equivalence of norms follows from Lemma \ref{Emb}.
\end{proof}

Observe that in view of Lemma \ref{defof} and Lemma \ref{Emb}, $\cV$ is continuously embedded in $L^\Phi$.\\
\indent We introduce a norm in $\V\times\W$ by the formula
\[
\|(v,w)\|:=\bigl(\|v\|_{\cD}^2+|w|_{\Phi}^2\bigr)^{\frac{1}{2}}
\]
and consider the energy functional defined by \eqref{eq:action} on $\cD(\curl,\Phi)$,
and 
\begin{equation}\label{eqJ}
\J(v,w):=\frac{1}{2}\int_{\R^3}|\nabla v|^2\,dx-\int_{\R^3}F(x,v+w)\,dx.
\end{equation}
defined on $\V\times\W$.
We have that $\J'$ is well defined and $\J$ is of class $\cC^1$ due to the following lemma.

\begin{Lem}\label{forC1}
If $u\in L^\Phi$, then 
$$\int_{\R^3}\Psi\bigl(\Phi'(|u|)\bigr)\,dx
\le C\int_{\R^3}\Phi(|u|)\,dx<\infty$$
for some constant $C>0$.
\end{Lem}
\begin{proof}
Since $\Phi\in\Delta_2$, it follows using Lemma \ref{AllProp}(i) and recalling $\Psi'(t)=\inf\{s\ge 0:\Phi'(s)>t\}$ that
\[
\int_{\R^3}\Psi\bigl(\Phi'(|u|)\bigr)\,dx \le \frac1{K'}\int_{\R^3}\Phi'(|u|)\Psi'(\Phi'(|u|))\,dx \le \frac1{K'}\int_{\R^3}\Phi'(|u|)|u|\,dx \le \frac K{K'}\int_{\R^3}\Phi(|u|)\,dx.
\]
\end{proof}

\begin{Prop} \label{classC1}
$\J'$ is well defined and $\J$ is of class $C^1$.
\end{Prop}

\begin{proof}
First we see that for every $v$, $v'\in\V$ and $w$, $w'\in\W$ there holds
\[
\begin{split}
\bigg|\int_{\R^3}\langle f(x,v+w) & ,v'+w'\rangle\,dx\bigg|\le\int_{\R^3}|f(x,v+w)||v'+w'|\,dx\le\\
& \le c_1\int_{\R^3}\Phi'(|v+w|)|v'+w'|\,dx\le C|\Phi'(|v+w|)|_\Psi|v'+w'|_\Phi
\end{split}
\]
for some $C>0$ by Lemma \ref{AllProp} (ii) and because $|\Phi'(|v+w|)|_\Psi<+\infty$ according to Lemma \ref{forC1}. Now we can use the argument of \cite[Lemma 2.1]{Clement} to show that $I\in C^1(L^\Phi,\R)$ where $I(v+w) := \int_{\R^3}F(x,v+w)\,dx$. Employing Lemma \ref{Emb}, it follows that $\J\in C^1(\cV\times \cW,\R)$.
\end{proof}

\begin{Prop}\label{PropSolutE}
Let $E=v+w\in\V\oplus\W$. Then $(v,w)$ is a critical point of $\J$ if and only if $E$ is a critical point of $\cE$ if and only if $\E$ is a weak solution to \eqref{eq}, i.e. $$\int_{\R^3}\langle E,\nabla\times\nabla\times\vp\rangle\,dx=\int_{\R^3}\langle f(x,E),\vp\rangle\,dx\quad\hbox{ for any }\vp\in\cC^\infty_0(\R^3,\R^3).$$
\end{Prop}
\begin{proof}
For the first equivalence, let $E'=v'+w'\in\V\oplus\W$. Then we have
\[
\int_{\R^3}\langle f(x,v+w),v'+w'\rangle\,dx=\int_{\R^3}\langle f(x,E),E'\rangle\,dx
\]
and, since $\curlop w=\curlop w'=0$,
\[
\int_{\R^3}\langle\curlop v,\curlop v'\rangle\,dx=\int_{\R^3}\langle\curlop E,\curlop E'\rangle\,dx
\]
so that
\[
\int_{\R^3}\langle\curlop v,\curlop v'\rangle\,dx=\int_{\R^3}\langle f(x,v+w),v'+w'\rangle\,dx\Leftrightarrow\int_{\R^3}\langle\curlop E,\curlop E'\rangle\,dx=\int_{\R^3}\langle f(x,E),E'\rangle\,dx
\]
and the conclusion follows from Lemma \ref{defof}.
For the second equivalence we just need to observe that for every $\vp\in\cC^\infty_0(\R^3,\R^3)$
\[
\int_{\R^3}\langle\curlop E,\curlop\vp\rangle\,dx=\int_{\R^3}\langle E,\curlop\curlop\vp\rangle\,dx.
\]
\end{proof}

\section{Critical point theory}\label{sec:criticaslpoitth}

We recall the abstract setting from \cite{BartschMederskiJFA,BartschMederski}.
Let $X$ be a reflexive Banach space with the norm $\|\cdot\|$ and a topological direct sum decomposition $X=X^+\oplus\tX$, where $X^+$ is a Hilbert space with a scalar product $\langle .\,,.\rangle$. For $u\in X$ we denote by $u^+\in X^+$ and $\tu\in\tX$ the corresponding summands so that $u = u^++\tu$. We may assume $\langle u,u \rangle = \|u\|^2$ for any
$u\in X^+$ and $\|u\|^2 = \|u^+\|^2+\|\tu \|^2$. The topology $\cT$ on $X$ is defined as the product of the norm topology in $X^+$ and the weak topology in $\tX$. Thus $u_n\cTto u$ is equivalent to $u_n^+\to u^+$ and $\tu_n\weakto\tu$.

Let $\cJ$ be a functional on $X$ of the form 
\begin{equation}\label{EqJ}
\cJ(u) = \frac12\|u^+\|^2-\cI(u) \quad\text{for $u=u^++\tu \in X^+\oplus \tX$}.
\end{equation}
The set
\begin{equation}\label{eq:ConstraintM}
\cM := \{u\in X:\, \cJ'(u)|_{\tX}=0\}=\{u\in X:\, \cI'(u)|_{\tX}=0\}
\end{equation}
obviously contains all critical points of $\cJ$. Suppose the following assumptions hold.
\begin{itemize}
	\item[(I1)] $\cI\in\cC^1(X,\R)$ and $\cI(u)\ge \cI(0)=0$ for any $u\in X$.
	\item[(I2)] $\cI$ is $\cT$-sequentially lower semicontinuous:
	$u_n\cTto u\quad\Longrightarrow\quad \liminf \cI(u_n)\ge \cI(u)$.
	\item[(I3)] If $u_n\cTto u$ and $\cI(u_n)\to \cI(u)$ then $u_n\to u$.
	\item[(I4)] $\|u^+\|+\cI(u)\to\infty$ as $\|u\|\to\infty$.
	\item[(I5)] If $u\in\cM$ then $\cI(u)<\cI(u+v)$ for every $v\in \tX\setminus\{0\}$.
\end{itemize}
Clearly, if a strictly convex functional $\cI$ satisfies (I4), then (I2) and (I5) hold. Observe that for any $u\in X^+$ we find  $m(u)\in\cM$ which is the unique global maximizer of $\J|_{u+\tX}$. Note that $m$ needs not be $\cC^1$, and $\cM$ needs not be a differentiable manifold because $\cI'$ is only required to be continuous. Recall from \cite{BartschMederskiJFA} that $(u_n)$ is called a $(PS)_c$-sequence for $\cJ$ if $\cJ'(u_n)\to 0$ and $\cJ(u_n)\to c$, and $\cJ$ satisfies the $(PS)_c^\cT$-condition on $\cM$ if each $(PS)_c$-sequence $(u_n)\subset\cM$ has a subsequence converging in the $\cT$-topology. In order to apply classical critical point theory like the mountain pass theorem to $\cJ\circ m: X^+\to \R$ we need some additional assumptions.
\begin{itemize}
	\item[(I6)] There exists $r>0$ such that $a:=\inf\limits_{u\in X^+,\|u\|=r} \cJ(u)>0$.
	\item[(I7)] $\cI(t_nu_n)/t_n^2\to\infty$ if $t_n\to\infty$ and $u_n^+\to u^+\ne 0$ as $n\to\infty$.
\end{itemize}

According to \cite[Theorem 4.4]{BartschMederskiJFA}, if (I1)--(I7) hold and 
\[
c_\cM := \inf_{\ga\in\Ga}\sup_{t\in [0,1]} \cJ(\ga(t)),
\]
where
\[
\Ga := \{\ga\in\cC([0,1],\cM):\ga(0)=0,\ \|\ga(1)^+\|>r \text{ and } \cJ(\ga(1))<0\},
\]
then  $c_{\cM}\ge a>0$ and $\cJ$ has a $(PS)_{c_\cM}$-sequence $(u_n)$ in $\cM$. If, in addition,
$\cJ$ satisfies the $(PS)_{c_\cM}^\cT$-condition in $\cM$,  then $c_\cM$ is achieved by a critical point of $\cJ$. Since we look for solutions to \eqref{eq} in $\R^3$ and not in a bounded domain as in \cite{BartschMederskiJFA}, the $(PS)_{c_\cM}^\cT$-condition is no longer satisfied. We consider the set
\begin{equation}\label{eq:NehariDef}
\cN := \{u\in X\setminus\tX: \cJ'(u)|_{\R u\oplus \tX}=0\} = \{u\in\cM\setminus\tX: \cJ'(u)[u]=0\} \subset\cM
\end{equation}
and we require the following condition on $\cI$:
\begin{itemize}
	\item[(I8)] $\frac{t^2-1}{2}\cI'(u)[u]+\cI(u) - \cI(tu+v)=\frac{t^2-1}{2}\cI'(u)[u] + t\cI'(u)[v] + \cI(u) - \cI(tu+v) \leq 0$\\ for every $u\in \cN$, $t\ge 0$, $v\in \tX$.
\end{itemize}
In \cite{BartschMederski,BartschMederskiJFA} it was additionally assumed that \emph{strict inequality holds} provided $u\neq tu+v$. This stronger variant of (I8)
implies that for any $u^+\in X^+\setminus\{0\}$ the functional $\cJ$ has a unique critical point $n(u^+)$ on the half-space $\R^+u^+ +\tX$. Moreover, $n(u^+)$ is the global maximizer of $\cJ$ on this half-space, the map
$$n:SX^+=\{u^+\in X^+: \|u^+\|=1\} \to \cN$$
is a homeomorphism, the set $\cN$ is a topological manifold, and it is enough to look for critical points of $\cJ\circ n$. $\cN$ is called the Nehari-Pankov manifold. This is the approach of \cite{SzulkinWethHandbook}. However, if the weaker condition (I8) holds, this procedure cannot be repeated. In particular, $\cN$ need not be a manifold. Yet the following holds.

\begin{Lem} \label{max}
If $u\in\cN$, then $u$ is a (not necessarily unique) maximizer of $\cJ$ on  $\R^+u +\tX$.
\end{Lem}

\begin{proof}
Let $u\in\cN$. In view of (I8) we get by explicit computation
	\begin{equation*}\label{eq:JonN}
	\cJ(tu+v)=\cJ(tu+v)-\cJ'(u)\Big[\frac{t^2-1}{2}u+tv\Big]\leq \cJ(u)
	\end{equation*}
	for any $t\geq 0$ and $v\in\tX$.
Hence  the conclusion.
\end{proof}

Let 
\[
\tJ := \cJ\circ m: X^+\to\R. 
\]
Before proving the main results of this section we recall the following properties (i)--(iv) taken from \cite[Proof of Theorem~4.4]{BartschMederskiJFA}. Note that (I8) has not been used there.
	\begin{itemize}
		\item[(i)] For each $u^+\in X^+$ there exists a unique $\tu\in \tX$ such that $m(u^+):=u^++\tu\in\cM$. This $m(u^+)$ is the minimizer of $\cI$ on $u^++\tX$.
		\item[(ii)] $m:X^+\to \cM$  is a homeomorphism with the inverse $\cM\ni u\mapsto u^+\in X^+$.
		\item[(iii)] $\wt{\cJ}=\cJ\circ m \in\cC^1(X^+,\R)$.
		\item[(iv)]$\wt{\cJ}'(u^+) = \cJ'(m(u^+))|_{X^+}:X^+\to\R$ for every $u^+\in X^+$.
	\end{itemize}
Property (i) has in fact already been discussed above. We shall also need the following fact. 

\begin{Lem} \label{infty}
	Let $X_k$ be a $k$-dimensional subspace of $X^+$. Then $\wt{\cJ}(u)\to-\infty$ whenever $\|u\|\to\infty$ and $u\in X_k$.
\end{Lem}

\begin{proof}
	It suffices to show that each sequence $(u_n^+)\subset X_k$ such that $\|u_n^+\|\to\infty$ contains a subsequence along which $\wt{\cJ}\to-\infty$.  Let $u_n^+ = t_nv_n$,  $\|v_n\|=1$ and $m(u_n^+)=u_n^++\tu_n\in\cM$. Then, passing to a subsequence and using (I7), we obtain
	\[
	\frac{\wt{\cJ}(t_nv_n)}{t_n^2} = \frac12 - \frac{\cI(t_n(v_n+\frac1{t_n}\tu_n))}{t_n^2} \to - \infty
	\]
	as claimed.
\end{proof}

As usual, $(u_n)\subset X^+$ will be called \emph{a Cerami sequence} for $\tJ$ at the level $c$ if $(1+\|u_n\|)\tJ'(u_n) \to 0$ and $\tJ(u_n) \to c$. In view of (I4), it is clear that if $(u_n)$ is a bounded Cerami sequence for $\tJ$, then $(m(u_n))\subset \cM$ is a bounded Cerami sequence for $\cJ$.

\begin{Th}\label{ThLink1}
	Suppose $\cJ \in \cC^1(X,\R)$ satisfies (I1)--(I8). Then:
	\begin{itemize}
		\item[(a)] $c_{\cM}\ge a>0$ and $\tJ$ has a Cerami sequence $(u_n)$ at the level $c_\cM$. 
		\item[(b)] $c_{\cM}=c_{\cN}:= \inf_\cN \cJ$.
	\end{itemize}
\end{Th}

The set $\cN_0 := \{u\in X^+\setminus\{0\}: \tJ'(u)[u]=0\}$ is called \emph{the Nehari manifold} for $\tJ$. Denote $c_{\cN_0} := \inf_{\cN_0}\cJ$.

\begin{altproof}{Theorem \ref{ThLink1}}
	Set
\begin{equation}\label{eq:Gamma_Phi}
	\wt{\Gamma} := \{\sigma\in\cC([0,1],X^+):\, \sigma(0)=0,\ \|\sigma(1)\|>r \hbox{ and }
	\tJ(\sigma(1))<0\}.
\end{equation}
	Observe that $\tJ$ has the mountain pass geometry and $\Gamma,\wt\Gamma$ are related as follows: if $\gamma\in\Gamma$, then $\gamma^+\in\wt\Gamma$ and $\cJ(\gamma(t))=\tJ(\gamma^+(t))$, and if $\sigma\in\wt\Gamma$, then $m\circ\sigma\in\Gamma$ and $\tJ(\sigma(t)) = \cJ(m\circ\sigma(t))$. Hence the mountain pass value for $\tJ$ is given by
	\begin{equation} \label{cm}
	c_{\cM}=\inf_{\sigma\in\wt{\Gamma}}\sup_{t\in [0,1]}\cJ\circ m(\sigma(t)) \equiv \wt{\cJ}(\sigma(t)) \ge a > 0.
	\end{equation} 
By the mountain pass theorem there exists a Cerami sequence $(u_n)$ for $\tJ$ at the level $c_\cM$ (see \cite{Cerami,bbf}) which proves (a). \\
\indent 			
The map $u\mapsto m(u)$ is a homeomorphism between $\cN_0$ and $\cN$, and since $\tJ(u)=\cJ(m(u))$, $c_{\cN_0}=c_{\cN}$. For $u\in X^+\setminus\{0\}$, consider $\wt{\cJ}(tu)$, $t>0$. By Lemma \ref{infty}, $\wt{\cJ}(tu)\to-\infty$ as $t\to\infty$. Hence $\max_{t > 0}\wt{\cJ}(tu)\geq a$ exists. If $t_1u,t_2u\in \cN_0$, then $m(t_1u),m(t_2u)\in\cN$, so by Lemma \ref{max}, $\wt{\cJ}(t_1u) =\wt{\cJ}(t_2u)$. Consequently, there exist $0 < t_{min} \le t_{max}$ such that $\wt{\cJ}(tu)\in\cN_0$ if and only if $t\in[t_{min},t_{max}]$ and $\wt{\cJ}(tu)$ has the same value for those $t$. Hence $\wt{\cJ}'(tu)[u]>0$ for $0<t<t_{min}$ and $\wt{\cJ}'(tu)[u]<0$ for $t>t_{max}$.  It follows that $X^+\setminus \cN_0$ consists of two connected components and therefore each path in $\wt{\Gamma}$ must intersect $\cN_0$. Therefore $c_{\cM}\ge c_{\cN_0}$. Since $c_{\cN_0} = \inf_{u\in X^+\setminus\{0\}}\max_{t>0}\tJ(tu)$, \eqref{cm} implies $c_{\cM} = c_{\cN_0} = c_{\cN}$. Note in particular that $\wt{\cJ}\ge 0$ on $B(0,r)$, where $r$ is given in (I6), so the condition $\|\sigma(1)\|>r$ in the definition of $\wt\Gamma$ is redundant because it must necessarily hold if $\tJ(\sigma(1))<0$.
\end{altproof}

Since $c_{\cN_0}=c_{\cN}=c_{\M}>0$, $\cN_0$ is bounded away from 0 and hence closed in $X^+$ while $\cN$ is bounded away from $\tX$ and hence closed in $X$.

For a topological group acting on $X$, denote \emph{the orbit of $u\in X$} by $G\ast u$, i.e., 
$$G\ast u:=\{gu: g\in G\}.$$
A set $A\subset X$ is called \emph{$G$-invariant} if $gA\subset A$ for all $g\in G$. $\cJ: X\to\R$ is called \emph{$G$-invariant} and $T: X\to X^*$ \emph{$G$-equivariant} if $\cJ(gu)=\cJ(u)$ and $T(gu)=gT(u)$ for all $g\in G$, $u\in X$. \\
\indent   
In order to deal with  multiplicity of critical points, assume that $G$ is a topological group such that
\begin{itemize}
	\item[(G)] $G$ acts on $X$ by isometries and discretely in the sense that for each $u\ne 0$, $(G*u)\setminus\{u\}$ is bounded away from $u$. Moreover, $\cJ$ is $G$-invariant and $X^+,\wt X$ are $G$-invariant.
\end{itemize}
Observe that $\cM$ is $G$-invariant and $m:X^+\to\cM$ is $G$-equivariant.  In our application to \eqref{eq} we have $G=\mathbb{Z}^3$ acting by translations, see Theorem \ref{ThMain}.

\begin{Lem} \label{discrete}
For all $u,v\in X$ there exists $\eps=\eps_{u,v}>0$ such that $\|gu-hv\|>\eps$ unless $gu=hv$ ($g,h\in G$).
\end{Lem}

\begin{proof}
Suppose $G*u\ne G*v$ (the other case is obvious). We may assume without loss of generality that $u\ne 0$ and $v$ minimizes the distance from $u$ to $G*v$.  Now it suffices to take $\eps := \frac12\|u-v\|$.
\end{proof}

We shall use the notation 
\begin{gather*}
\tJ^\beta := \{u\in X^+: \tJ(u)\le\beta\}, \quad  \tJ_\alpha := \{u\in X^+:  \tJ(u)\ge\alpha\}, \\
 \tJ_\alpha^\beta := \tJ_\alpha\cap\tJ^\beta, \quad \cK:=\big\{u\in X^+: \tJ'(u)=0\big\}.
\end{gather*}
Since all nontrivial critical points of $\cJ$ are in $\cN$, it follows from Theorem \ref{ThLink1} that $\tJ(u)\ge a$ for all $u\in\cK\setminus\{0\}$. \\
\indent 
We introduce the following variant of the {\em Cerami condition} between the levels $\alpha, \beta\in\R$.
\begin{itemize}
	\item[$(M)_\alpha^\beta$]
	\begin{itemize}
		\item[(a)] Let $\alpha\le\beta$. There exists $M_\alpha^\beta$ such that $\limsup_{n\to\infty}\|u_n\|\le M_\alpha^{\beta}$ for every $(u_n)\subset X^+$ satisfying $\alpha\le\liminf_{n\to\infty}\tJ(u_n)\le\limsup_{n\to\infty}\tJ(u_n)\leq\beta$ and \linebreak $(1+\|u_n\|)\tJ'(u_n)\to 0$.
		\item[(b)] Suppose in addition that the number of critical orbits in $\tJ_\alpha^\beta$ is finite. Then there exists $m_\alpha^\beta>0$ such that if $(u_n), (v_n)$ are two sequences as above and $\|u_n-v_n\|<m_\alpha^\beta$ for all $n$ large, then  $\liminf_{n\to\infty}\|u_n-v_n\|=0$.
	\end{itemize}
\end{itemize}

Note that if $\cJ$ is even, then $m$ is odd (hence $\tJ$ is even) and $\cM$ is symmetric,  i.e. $\cM=-\cM$. Note also that $(M)_\alpha^\beta$ is a condition on $\tJ$ and \emph{not} on $\cJ$. Our main multiplicity result reads as follows.
\begin{Th}\label{Th:CrticMulti}
		Suppose $\cJ \in \cC^1(X,\R)$ satisfies (I1)--(I8) and $\dim(X^+)=\infty$.\\
		(a) If  $(M)_0^{c_\cM+\eps}$ holds for some $\eps>0$, then either $c_{\cM}$ is attained by a critical point or there exists a sequence of critical values $c_n$ such that $c_n>c_{\cM}$ and $c_n\to c_{\cM}$ as $n\to\infty$.\\		
		 (b) If $(M)_0^{\beta}$ holds for every $\beta>0$ and $\cJ$ is even, 
	then $\cJ$ has infinitely many distinct critical orbits.
\end{Th}

By a standard argument we can find
	a locally Lipschitz continuous pseudo-gradient vector field $v:X^+\setminus \cK\to X^+$ associated with $\wt{\cJ}$, i.e.
	\begin{eqnarray}
	\|v(u)\|&<&1,\label{eq:flow2}\\
	\wt{\cJ}'(u)[v(u)] &>& \frac12\|\wt{\cJ}'(u)\|\label{eq:flow3}
	\end{eqnarray}
	for any $u\in X^+\setminus \cK$. Moreover, if $\J$ is even, then $v$ is odd.  
	Let
	$\eta:\cG\to  X^+\setminus \cK$ be the flow defined by
	\begin{equation*}
	\left\{
	\begin{aligned}
	&\partial_t \eta(t,u)=-v(\eta(t,u))\\
	&\eta(0,u)=u
	\end{aligned}
	\right.
	\end{equation*}
	where $\cG:=\{(t,u)\in [0,\infty)\times (X^+\setminus \cK):\; t<T(u)\}$ and $T(u)$ is the maximal time of existence of $\eta(\cdot,u)$.  We prove Theorem \ref{Th:CrticMulti} by contradiction. From now on we assume:  
\[
\text{There is a finite number of distinct orbits $\{G\ast u: u\in \cK\}$.}
\]
	
\begin{Lem}\label{lem:flow}
	Suppose $(M)_0^\beta$ holds for some $\beta>0$ and let $u \in \tJ_0^\beta\setminus \cK$. Then either $\lim_{t \to T(u)}\eta(t,u)$ exists and is a critical point of $\wt{\cJ}$ or $\lim_{t\to T(u)}\wt{\cJ}(\eta(t,u)) = -\infty$. In the latter case $T(u)=\infty$.
\end{Lem}

\begin{proof}
	Suppose $T(u)<\infty$ and let $0\le s<t<T(u)$. Then
	\[
	\|\eta(t,u)-\eta(s,u)\| \le \int_s^t \|v(\eta(\tau,u))\|\,d\tau \le t-s.
	\] 
	Hence the limit exists and if it is not a critical point, then $\eta(\cdot,u)$ can be continued for $t>T(u)$. \\
	\indent Suppose now $T(u)=\infty$ and $\wt{\cJ}(\eta(t,u))$ is bounded from below. We distinguish three cases: 
	\begin{itemize}
		\item[(i)] $t\mapsto \eta(t,u)$ is bounded,
		\item[(ii)] $t\mapsto \eta(t,u)$ is unbounded but $\|\eta(t,u)\|\not\to\infty$,
		\item[(iii)] $\|\eta(t,u)\|\to\infty$.
	\end{itemize} 
	 (i) We follow an argument in \cite{SzulkinWeth}. We shall show that for each $\eps>0$ there exists $t_\eps>0$ such that
	$\|\eta(t_\eps,u)-\eta(t,u)\|<\eps$ for all $t \ge t_\eps$ (this implies $\lim_{t \to \infty}\eta(t,u)$ exists, and then it is obviously a critical point). Arguing by contradiction, we can find $\eps\in(0,m_0^\beta/2)$, $R>0$ and $t_n\to\infty$ such that $\eta(t_n,u)\in  B(0,R)$ and $\|\eta(t_n,u)-\eta(t_{n+1},u)\| = \eps$ for all $n$.  Let $t_n^1$ be the smallest $t\in(t_n,t_{n+1})$ such that $\|\eta(t_n,u)-\eta(t_n^1,u)\|= \eps/3$ and $t_n^2$ the largest $t\in(t_n^1,t_{n+1})$ such that $\|\eta(t_{n+1},u)-\eta(t_n^2,u)\|= \eps/3$. Put $\kappa_n := \min \{\|\wt{\cJ}'(\eta(t,u))\|: t\in[t_n,t_n^1]\}$. Then
	\begin{eqnarray*}
		\frac{\eps}3 & = & \|\eta(t_n^1,u)-\eta(t_n,u)\| \le \int_{t_{n}}^{t_n^1}  \|v(\eta(t,u))\| \,dt \le t_n^1-t_n \\ 
		& \le & \frac{2}{\kappa_n} \int_{t_n}^{t_n^1} \wt{\cJ}'(\eta(t,u))[v(\eta(t,u))]\,dt = \frac2{\kappa_n}\left(\wt{\cJ}(\eta(t_n,u))-\wt{\cJ}(\eta(t_n^1,u))\right).
	\end{eqnarray*}
	Since $\wt{\cJ}(\eta(t_n,u))-\wt{\cJ}(\eta(t_n^1,u))\to 0$, also $\kappa_n\to 0$. Hence we can choose $s_n^1\in[t_n,t_n^1]$ such that if $\tu_n:= \eta(s_n^1,u)$, then $\wt{\cJ}'(\tu_n)\to 0$. As  $\|\eta(s_n^1,u)\|$ is bounded, $(\tu_n)$ is a Cerami sequence. A similar argument shows the existence of $\tv_n:= \eta(s_n^2,u)$ ($s_n^2\in[t_n^2,t_{n+1}]$) such that $\wt{\cJ}'(\tv_n)\to 0$. Hence
	\[
	\frac\eps 3 \le \liminf_{n\to\infty}\|\tu_n-\tv_n\|\le \limsup_{n\to\infty}\|\tu_n-\tv_n\|\le \eps +\frac{2}{3}\eps<m_0^\beta,
	\]
	a contradiction to $(M)_0^\beta (b)$.\\
\indent			
	(ii) Observe that  there are no Cerami sequences in $X^+\setminus B(0,M_0^\beta)$ at any level $\alpha\in[0, \beta]$ according to $(M)_0^\beta (a)$. Since $\eta(t,u)$ is unbounded but $\|\eta(t,u)\|\not\to\infty$, we can find $R>M_0^\beta$ such that there exist arbitrarily large $t$ for which $\eta(t,u)\in B(0,R)$. We can find $t_n, t_n^1$ so that $t_n\to\infty$, $\|\eta(t_n,u)\| = R+1$ and $t_n^1$ is the smallest $t>t_n$ with $\|\eta(t_n^1,u)\| = R$. We may also assume that $\|\eta(s,u)\|\leq R+1$ for $s\in [t_n,t_n^1]$. Let $\kappa_n$ be as above. Then
	\[
	1\le \|\eta(t_n^1,u)-\eta(t_n,u)\| \le \frac2{\kappa_n}\left(\wt{\cJ}(\eta(t_n,u))-\wt{\cJ}(\eta(t_n^1,u))\right)
	\]
	and hence $\kappa_n\to 0$. So we see that there exist $\tu_n:= \eta(s_n^1,u)$, $s_n^1\in[t_n,t_n^1]$, such that  $R\le \|\wt u_n\|\le R+1$ and $\wt{\cJ}'(\tu_n)\to 0$. Thus we have found a Cerami sequence in $X^+\setminus B(0,M_0^\beta)$ which is impossible. This shows that case (ii) can never occur.\\
\indent
	(iii) There exist $R_0>0$ and $\delta>0$ such that $\|\wt{\cJ}'(v)\|\ge \delta/\|v\|$ whenever $\|v\|\ge R_0$ and $v\in \tJ_0^\beta$ (for otherwise there exists an unbounded Cerami sequence). Choose $t_0>0$ so that $\|\eta(t,u)\|\ge R_0$ and $\wt{\cJ}(\eta(t_0,u))-\wt{\cJ}(\eta(t,u)) \le \delta/8$ for $t\ge t_0$. For large $n$ let $t_n$ be the smallest $t$ such that $\|\eta(t,u)\|=n$, and let $\kappa_n := \min \{\|\wt{\cJ}'(\eta(t,u))\|: t\in[t_0,t_n]\}$. By the choice of $t_n$,
	\[
	\kappa_n \ge \min_{t\in [t_0,t_n]} \frac{\delta}{\|\eta(t,u)\|} = \frac{\delta}{\|\eta(t_n,u)\|}.
	\]
	It follows by the same argument as above that for $n$ large enough,
	\begin{eqnarray*}
		\frac12 \|\eta(t_n,u)\| & \le & \|\eta(t_n,u)-\eta(t_0,u)\| \le \frac2{\kappa_n}\left(\wt{\cJ}(\eta(t_0,u))-\wt{\cJ}(\eta(t_n,u))\right) \\
		& \le & \frac{2}{\delta}\|\eta(t_n,u)\|\left(\wt{\cJ}(\eta(t_0,u))-\wt{\cJ}(\eta(t_n,u))\right).
	\end{eqnarray*}
	This is a contradiction and hence also case (iii) can be ruled out. 
\end{proof}

Let $\Sigma := \{A\subset X^+: A=-A \text{ and } A \text{ is compact}\}$,
\[
\cH := \{h: X^+\to X^+ \text{ is a homeomorphism, }  h(-u)=-h(u) \text{ and } \wt{\cJ}(h(u))\le \wt{\cJ}(u) \text{ for all } u\},
\]
and for $A\in\Sigma$, put
\[
i^*(A) := \min_{h\in\cH} \gamma(h(A)\cap S(0,r))
\]
where $r$ is as in (I6), $S(0,r):=\{u\in X^+:\|u\|=r\}$ and $\gamma$ is Krasnoselskii's genus \cite{Struwe}. This is a variant of Benci's  pseudoindex \cite{bbf, be} and the following properties are adapted from \cite[Lemma 2.16]{SquassinaSzulkin}.

\begin{Lem} \label{index}
	Let $A,B\in\Sigma$.\\
	(i) If $A\subset B$, then $i^*(A)\le i^*(B)$. \\
	(ii) $i^*(A\cup B)\le i^*(A)+\gamma(B)$. \\
	(iii) If $g\in \cH$, then $i^*(A)\le i^*(g(A))$. \\
	(iv) Let $X_k$ be a $k$-dimensional subspace of $X^+$. Then $i^*(X_k\cap\overline B(0,R))\ge k$ whenever $R$ is large enough and $\overline B(0,R):=\{u\in X^+: \|u\|\leq R\}$.
\end{Lem}

\begin{proof}
	(i) follows immediately from the properties of genus. \\
\indent	
	(ii) For each $h\in\cH$,
	\[
	i^*(A\cup B) \le \gamma(h(A\cup B)\cap S(0,r)) = \gamma ((h(A)\cup h(B))\cap S(0,r)) \le \gamma(h(A)\cap S(0,r)) + \gamma(B).
	\]
	Taking the minimum over all $h\in\cH$ on the right-hand side we obtain the conclusion. \\
	\indent
	(iii) Since $\wt{\cJ} (g(u))\le \wt{\cJ}(u)$ for all $u\in X^+$, $h\circ g\in\cH$ if $h\in\cH$. Hence $\{h\circ g: h\in\cH\}\subset \cH$ and therefore
	\[
	\min_{h\in\cH}\gamma(h(A)\cap S(0,r)) \le \min_{h\in\cH}\gamma((h\circ g)(A)\cap S(0,r)).
	\]
\indent 	(iv) By Lemma \ref{infty}, $\wt{\cJ}(u)<0$ on $X_k\setminus B(0,R)$ if $R$ is large enough. Let $D := X_k\cap\overline B(0,R)$. Suppose $i^*(D)<k$, choose $h\in\cH$ such that $\gamma(h(D)\cap S(0,r))<k$ and an odd mapping 
	$$f: h(D)\cap S(0,r) \to \R^{k-1}\setminus\{0\}.$$ 
	Let $U:= h^{-1}(B(0,r))\cap X_k$. Since $\wt{\cJ}(h(u))\le \wt{\cJ}(u)< 0$ for $u\in X_k\setminus B(0,R)$ and $\wt{\cJ}(u)\ge 0$ for $u\in B(0,r)$, it follows that $U\subset D\setminus\partial D$ and hence $U$ is an open and bounded neighbourhood of 0 in $X_k$. If $u\in \partial U$, then $h(u)\in S(0,r)$ and therefore $f\circ h: \partial U \to \R^{k-1}\setminus\{0\}$, contradicting the Borsuk-Ulam theorem \cite[Proposition II.5.2]{Struwe}, \cite[Theorem D.17]{Willem}. So $i^*(D)\ge k$.  
\end{proof}

\begin{altproof}{Theorem \ref{Th:CrticMulti}}
(a) Suppose that $\wt{\cJ}$ has no critical values in $[c_{\cM},c_{\cM}+\eps_0]$ for some $\eps_0\in (0,\eps]$. Thus $\wt{\cJ}$ has only the trivial critical point 0 in $\wt{\cJ}^{c_{\cM}+\eps_0}$.
Take $u\in \wt{\cJ}^{c_{\cM}+\eps_0}$ and
observe that by Lemma \ref{lem:flow}, either $\lim_{t \to T(u)}\eta(t,u)=0$ or $\lim_{t\to T(u)}\wt{\cJ}(\eta(t,u)) = -\infty$. 
Hence we may define the entrance time map $e:\wt{\cJ}^{c_{\cM}+\eps_0}\to [0,\infty)$ by the formula
$$e(u):=\inf\{t\in [0,T(u)):  \wt{\cJ}(\eta(t,u))\leq c_{\cM}/2\}.$$ 
Take any $\gamma\in\wt{\Gamma}$ such that 
$$\wt{\cJ}(\gamma(t))=\J(m(\gamma(t)))<c_{\cM}+\eps_0 \quad \text{for all } t\in[0,1]$$
where $\wt{\Gamma}$ is given by \eqref{eq:Gamma_Phi}. Since $e$ is continuous,
$\tilde{\gamma}(t):=\eta\big(e(\gamma(t)),\gamma(t)\big)$ is a continuous path in $X^+$ such that $\wt{\cJ}(\tilde{\gamma}(1))\leq \wt{\cJ}(\gamma(1))<0$. Hence $\tilde{\gamma}\in\wt{\Gamma}$ and 
$$c_{\cM}=\inf_{\sigma\in\wt{\Gamma}}\sup_{t\in [0,1]}\wt{\cJ}(\sigma(t))\leq \sup_{t\in [0,1]}\wt{\cJ}(\tilde{\gamma}(t))\leq c_{\cM}/2.$$
The obtained contradiction proves that either $c_{\cM}$ is a critical value or for any $\eps_0\in (0,\eps)$ we find a critical value in $(c_{\cM},c_{\cM}+\eps_0]$.\\
\indent
(b)	Take $\beta\geq a$ and let
	\begin{eqnarray*}
		\cK^\beta&:=&\{u\in \cK:  \wt{\cJ}(u)=\beta\}.
	\end{eqnarray*}
	Since there are finitely many critical orbits, there exists $\eps_0>0$ for which 
	\begin{equation}\label{eq:KL}
	\cK\cap \wt{\cJ}_{\beta-\eps_0}^{\beta+\eps_0}=\cK^\beta.
	\end{equation} 
	Choose $\delta\in (0,m_0^{\beta+\eps_0})$ such that $\overline B(u,\delta)\cap \overline B(v,\delta) = \emptyset$ for all $u,v\in\cK^\beta$, $u\ne v$ (this is possible due to Lemma \ref{discrete}). We show there is $\eps\in (0,\eps_0)$ such that
	\begin{equation}\label{eq:entrancetime1}
	\lim_{t\to T(u)} \wt{\cJ}(\eta(t,u)) < \beta -\eps \quad\hbox{for } u\in \wt{\cJ}^{\beta+\eps}_{\beta-\eps}\setminus B(\cK^\beta,\delta).
	\end{equation}
	We assume $\cK^\beta\ne\emptyset$, the other case being simpler.
	If $u\in \wt{\cJ}^{\beta+\eps_0}_{\beta-\eps_0}\setminus B(\cK^\beta,\delta)$ and $\lim_{t\to T(u)} \wt{\cJ}(\eta(t,u))\\ < \beta -\eps_0$, then \eqref{eq:entrancetime1} trivially holds. 
	Otherwise
	$$u\in A_0:=\big\{u\in \wt{\cJ}^{\beta+\eps_0}_{\beta-\eps_0}\setminus B(\cK^\beta,\delta): \lim_{t \to T(u)}\eta(t,u)\in \cK^\beta\big\}.$$ 
	Let $u\in A_0$ and define
	\begin{eqnarray*}
		t_0(u)&:=&\inf\big\{t\in [0,T(u)):  \eta(s,u)\in B(\cK^\beta,\delta)\hbox{ for all }s> t\big\},\\
		t(u)&:=&\inf\big\{t\in [t_0(u),T(u)):  \eta(t,u)\in B(\cK^\beta,\delta/2)\big\}
	\end{eqnarray*}
	and note that $0\leq t_0(u)< t(u)<T(u)$.  By \eqref{eq:flow2} we have
	\begin{eqnarray} \label{eq:t}
		\frac{\delta}{2} \leq
		\|\eta(t_0(u),u)-\eta(t(u),u)\|  \leq \int_{t_0(u)}^{t(u)}\|v(\eta(s,u))\|\, ds
		\leq t(u)-t_0(u). 
	\end{eqnarray}
	Let
	$$\rho:=\inf\{\|\wt{\cJ}'(\eta(t,u))\|: u\in A_0,\ t\in[t_0(u),t(u)]\}.$$
	If $\rho=0$ then we find $u_n\in A_0 $ and $t_n\in (t_0(u_n),t(u_n))$ such that $$\wt{\cJ}'(\eta(t_n,u_n))\to 0 \hbox{ as } n\to\infty.$$
	Since $t_n> t_0(u_n)$, we have $\eta(t_n,u_n)\in B(\cK^\beta,\delta)$ and passing to a subsequence we can find $u_0\in\cK^\beta$ and $g_n\in G$ such that
	$$g_n\eta(t_n,u_n)\in B(u_0,\delta).$$ Since $t_n<t(u_n)$, we see that
	$$g_{n}\eta(t_n,u_n)\notin B(\cK^\beta,\delta/2).$$
	Let $\tu_n:=u_0$, $\tv_n:=g_n\eta(t_n,u_n)$. Then $\tu_n$ and $\tv_n$ are two Cerami (in fact Palais-Smale)
	 sequences such that $\delta/2 \le \|\tv_n-\tu_n\| \le \delta<m_0^{\beta+\eps_0}$, a contradiction.
Therefore $\rho>0$ and we take 
	$$\eps < \min\Big\{\eps_0,\frac{\delta\rho}8\Big\}, \quad u\in \tJ_{\beta-\eps}^{\beta+\eps}\setminus B(\cK^\beta,\delta).$$ 
	Since
	\begin{eqnarray*}
		\wt{\cJ}(\eta(t(u),u))-\wt{\cJ}(\eta(t_0(u),u)) &=&
		-\int_{t_0(u)}^{t(u)}\wt{\cJ}'(\eta(s,u))[v(\eta(s,u))]\,ds\\
		&\leq& 
		-\frac12\int_{t_0(u)}^{t(u)}\|\wt{\cJ}'(\eta(s,u)\|\, ds,
	\end{eqnarray*}
	we obtain using \eqref{eq:t}
	\begin{eqnarray*}
		 \lim_{t\to T(u)} \wt{\cJ}(\eta(t,u))&\leq& \wt{\cJ}(\eta(t(u),u))
		\leq\beta +\eps -\frac12\int_{t_0(u)}^{t(u)}\|\wt{\cJ}'(\eta(s,u)\|\, ds\\
		&\leq& \beta+\eps -\frac{\delta\rho}4 < \beta-\eps.
	\end{eqnarray*}
	Hence $A_0=\emptyset$ which proves \eqref{eq:entrancetime1}. Note that this argument also shows $\eta(t,u)$ will not enter the set $B(\cK^\beta,\delta/2)$ if $u\in \tJ_{\beta-\eps}^{\beta+\eps}\setminus B(\cK^\beta,\delta)$. \\
	\indent
	Define 
	\[
	\beta_k :=  \inf_{i^*(A)\ge k} \sup_{u\in A}\wt{\cJ}(u), \quad k=1,2,\ldots.
	\]
	and note that by Lemma \ref{index}  all $\beta_k$ are well defined, finite and $a\le\beta_1\le\beta_2\le\ldots$.
	Let $\beta=\beta_k$  for some $k\geq 1$. If the set $\cK^\beta$ is nonempty, it is (at most) countable, so we can order its elements in pairs $\pm u_j$ and let the map $f:\cK^\beta\to\R\setminus\{0\}$ be given by $f(\pm u_j)=\pm 1$. This shows that by the choice of $\delta$,
	$$\gamma (\overline B(\cK^\beta,\delta))=\gamma (\cK^\beta)=1.$$
	Choose $\eps>0$ such that \eqref{eq:entrancetime1} holds. 
	Take Lipschitz continuous cutoff functions $\chi,\xi$ such that $\chi=0$ in $B(\cK^\beta,\delta/4)$, $\chi=1$ in $X^+\setminus B(\cK^\beta,\delta/2)$ and $\xi=1$ in $\wt{\cJ}_{\beta-\eps}^{\beta+\eps}$, $\xi=0$ in $X^+\setminus U$ , where $U$ is an open neighbourhood of $\wt{\cJ}_{\beta-\eps}^{\beta+\eps}$ with $\cK\cap U=\cK^\beta$.
	Let
	$\wt\eta:\R\times X^+\to  X^+$ be the flow given by
	\begin{equation*}
	\left\{
	\begin{aligned}
	&\partial_t \wt\eta(t,u)=-\chi(\wt\eta(t,u))\xi(\wt\eta(t,u)) v(\wt\eta(t,u))\\
	&\wt\eta(0,u)=u.
	\end{aligned}
	\right.
	\end{equation*}
	Then $\wt\eta(t,u)=\eta(t,u)$ as long as $t\ge 0$ and  $ \wt\eta(t,u)\in \wt{\cJ}^{\beta+\eps}_{\beta-\eps}\setminus B(\cK^\beta,\delta/2)$. 	Using  \eqref{eq:entrancetime1} we can define the entrance time map $e:\wt{\cJ}^{\beta+\eps}\setminus B(\cK^\beta,\delta)\to [0,\infty)$:
	$$e(u):=\inf\{t\in [0,\infty):  \wt{\cJ}(\wt\eta(s,u))\leq \beta -\eps\}.$$ 
	Since $\eta(s,u)\notin B(\cK^\beta,\delta/2)$ as we have observed, $e$ is finite. It is standard to show that $e$ is continuous and even. Take any $A\in\Sigma$ such that $i^*(A)\geq k$ and $\wt{\cJ}(u)\leq \beta+\eps$ for $u\in A$. Let $T:=\sup_{u\in A} e(u)$; then $T<\infty$ since $A$ is compact. Set $h:=\wt\eta(T,\cdot)$ and note that $h\in\cH$ and
	$$h(A\setminus B(\cK^\beta,\delta)) \subset \wt{\cJ}^{\beta-\eps}.$$
Therefore
	$$i^*(A\setminus B(\cK^\beta,\delta))\leq 
	i^*(h(A\setminus B(\cK^\beta,\delta)))\leq  k-1$$
	and
	\begin{equation}\label{eq:LSvaluse}
	k\leq i^*(A)\leq \gamma (\overline B(\cK^\beta,\delta)\cap A)+
	i^*(A\setminus B(\cK^\beta,\delta))
	\leq \gamma(\cK^\beta)+k-1.
	\end{equation}
	Thus $\cK^\beta\neq \emptyset$, so as we have shown above, $\gamma(\cK^\beta)=1$.	If $\beta_k=\beta_{k+1}$ for some
	$k\geq 1$, then \eqref{eq:LSvaluse} implies $\gamma(\cK^{\beta_k})\geq 2$, a contradiction. Hence we get an infinite sequence
	$\beta_1<\beta_2<...$
	of critical values which contradicts our assumption that $\cK$ consists of a finite number of distinct orbits. 
	This completes the proof.	
\end{altproof}

\section{Properties of the functional $J$ for curl-curl}

Recall our earlier assumption that (N1)--(N3) and (F1)--(F5) hold. We will check that assumptions (I1)--(I8) are satisfied and we want to apply Theorems \ref{ThLink1} and \ref{Th:CrticMulti}.\\
\indent Define the manifold
\begin{eqnarray}\label{DefOfN}
\mathcal{M} &:=& \{(v,w)\in\V\times\W:
\J'(v,w)[(0,\psi)]=0\,\hbox{ for any }\psi\in \W\}
\end{eqnarray}
and the Nehari-Pankov set for $\J$
\begin{eqnarray}\label{DefOfM}
\mathcal{N} &:=& \{(v,w)\in\V\times\W: u\neq 0,\; 
\J'(v,w)[(v,w)]=0 \\\nonumber
&&\hbox{ and }\J'(v,w)[(0,\psi)]=0\,\hbox{ for any }\psi\in \W\} 
\ \subset\ \cM. \nonumber
\end{eqnarray}
Observe that $E=v+w\in\cN_{\cE}$ if and only if $(v,w)\in\cN$ ($\cN_{\cE}$ is defined in \eqref{DefOfNehari1}). Moreover, $\mathcal{N}$ contains all nontrivial critical points of $\J$. In general  $\cN_{\cE}$, $\cN$ and $\cM$ are not $\cC^1$-manifolds.

\begin{Prop}\label{Propuv_N}
	If $(v,w)\in \V\times\W$ then
	$$\J(tv,tw+\psi)-\J'(v,w)\Big[\Big(\frac{t^2-1}{2}v,\frac{t^2-1}{2}w+t\psi\Big)\Big]\leq\J(v,w)$$
	for any $\psi\in\W$ and $t\geq 0$. 
\end{Prop}

\begin{proof}
	Let $(v,w)\in\V\times\W$, $\psi\in\W$, $t\geq 0$. We define
	$$D(t,\psi):=\J(t v,tw+\psi)-\J(v,w)-\J'(v,w)\Big[\Big(\frac{t^2-1}{2}v,\frac{t^2-1}{2}w+t\psi\Big)\Big]$$
and observe that
	\begin{eqnarray*}
		D(t,\psi)
		&=&\int_{\R^3}\langle f(x,v+w),\frac{t^2-1}{2}(v+w)+t\psi\rangle\, dx\\
		&&+\int_{\R^3}F(x,v+w)-F(x,t(v+w)+\psi)\,dx.
	\end{eqnarray*}
	For fixed $v,w\in\R^3$, define a map $\varphi:[0,+\infty)\times \R^3\to \R$ as follows:
	$$\varphi(t,\psi):=
	\langle f(x,v+w),\frac{t^2-1}{2}(v+w)+t\psi\rangle+F(x,v+w)-F(x,t(v+w)+\psi).$$
	We shall show that $\vp(t,\psi)\leq 0$ for all $t\geq 0$, $\psi\in\R^3$. This is clear if $v+w=0$. So let $v+w\ne 0$
	 and $\zeta := t(v+w)+\psi$.
	By (F3), (F4) we have
	$\varphi(0,\psi)\leq 0$ and 
	\begin{eqnarray*}
		\varphi(t,\psi) & \le & \langle f(x,v+w),\frac{t^2-1}{2}(v+w)+t(\zeta-t(v+w))\rangle+\frac12\langle f(x,v+w),v+w\rangle-F(x,\zeta) \\
		& = & -\frac12t^2\langle f(x,v+w),v+w\rangle + t\langle f(x,v+w),\zeta\rangle - A|\zeta|^2 + (A|\zeta|^2-F(x,\zeta)).
	\end{eqnarray*}
	If $A$ is large enough, then the quadratic form (in $t$ and $\zeta$) above is negative definite. Moreover, $A|\zeta|^2-F(x,\zeta)$ is bounded above by superquadraticity of $F$ implied by (F3) and (N3). Hence $\varphi(t,\psi)	\to -\infty$ as $t+|\psi|\to\infty$ and  $\varphi$ attains a maximum at some $(t,\psi)$ with $t\geq 0$.  If $t=0$, then $\vp(t,\psi)\leq 0$ as we have already mentioned. If $t>0$, then
	\begin{eqnarray}
	\partial_t\varphi(t,\psi) & = &  \langle f(x,v+w),t(v+w)+\psi \rangle - \langle f(x,t(v+w)+\psi),v+w\rangle = 0, \label{1st} \\
	\partial_\psi\varphi(t,\psi) &  = & tf(x,v+w) - f(x,t(v+w)+\psi) = 0. \label{2nd}
	\end{eqnarray} 
	Using \eqref{2nd} in \eqref{1st} we see that both terms in \eqref{1st} are positive (because $\langle f(x,v+w),v+w\rangle>0$) and $\langle f(x,v+w),\psi \rangle = 0$. This and (F5) imply
	$$
	\varphi(t,\psi) = \frac{t^2-1}{2}\langle f(x,v+w),v+w \rangle+F(x,v+w)-F(x,t(v+w)+\psi) \le 0.
	$$
\end{proof}

Consider $I:L^{\Phi}\to \R$ and  $\cI:L^{\Phi}\times \W\to\R$ given by
\begin{equation}\label{DefOfXi}
\cI(v,w):=I(v+w):=\int_{\R^3}F(x,v+w)\, dx\quad\hbox{ for }(v,w)\in L^{\Phi}\times\W.
\end{equation}
By Proposition \ref{classC1}, $I$ and $\cI$ are of class $\cC^1$. In view of (F2), $I$ and $\cI$ are strictly convex. Moreover, the following property holds.

\begin{Lem}\label{LemConvWeakIpliesStrong}
	If $E_n\rightharpoonup E$ in $L^{\Phi}$ and
	$I(E_n)\to I(E)$
	then $E_n\to E$ in $L^{\Phi}$.
\end{Lem}

Before proving the above lemma we need a variant of the Brezis-Lieb result \cite{BrezisLieb} for sequences in $L^{\Phi}$. 
\begin{Lem}\label{LemBrezLieb}
	Let $(E_n)$ be a bounded sequence in $L^{\Phi}$ such that
	$E_n\to E$ a.e. on $\R^3$.
	Then
	$$\lim_{n\to+\infty}\int_{\R^3}F(x,E_n)-F(x,E_n-E)\, dx=\int_{\R^3}F(x,E)\, dx.$$
\end{Lem}

\begin{proof}
	Note that
	\begin{eqnarray*}
		\int_{\R^3}F(x,E_n)-F(x,E_n-E)\, dx
		&=&\int_{\R^3}\int_0^1\frac{d}{dt}F(x,E_n-E+tE)\, dtdx\\
		&=&\int_0^1\int_{\R^3}\langle f(x,E_n-E+tE),E\rangle\, dxdt
	\end{eqnarray*}
	and $f(x,E_n-E+tE)$ is bounded
	in $L^{\Psi}$ according to (F3) and Lemmas \ref{AllProp} (iv), \ref{forC1}.
	Thus for any $\Omega\subset\R^3$,
	\begin{eqnarray} \label{*}
		\int_{\Omega}|\langle f(x,E_n-E+tE),E\rangle|\, dx
		\leq |f(x,E_n-E+tE)|_{\Psi}
		|E\chi_{\Omega}|_{\Phi}.
	\end{eqnarray}
	By \cite[Definition III.4.2, Corollary III.4.5 and Theorem III.4.14]{RaoRen} the space $L^\Phi$ has an absolutely continuous norm, so by \eqref{*}, for any $\eps>0$ there is $\delta>0$
	such that if $|\Omega|<\delta$ ($|\Om|$ denotes the measure of $\Om$), then
	$$\int_{\Omega}|\langle f(x,E_n-E+tE),E\rangle|\, dx  < \eps$$
	independently of $n$.
	Thus $(\langle f(x,E_n-E+tE),E\rangle )$ is uniformly integrable.
	Using \eqref{*} once more we see that for any $\eps>0$ there is $\Omega\subset\R^3$ with 
	$|\Omega|<+\infty$ such that
	$$\int_{\Omega^c}\langle f(x,E_n-E+tE),E\rangle\, dx < \eps.$$
	Indeed, if $\chi_n$ is the characteristic function of the set $|x|\ge n$, then $\int_{\R^3}\Phi(|E\chi_n|)\,dx \to 0$ and therefore $|E\chi_n|_\Phi\to 0$ by Lemma \ref{AllProp}(iii).
	Hence $\Om$ exists as claimed and $(\langle f(x,E_n-E+tE),E\rangle )$ is tight.
	Since $E_n(x)-E(x)\to 0$ a.e. on $\R^3$, it follows from the Vitali convergence theorem that
	$$\int_{\R^3}F(x,E_n)-F(x,E_n-E)\, dx\to
	\int_0^1\int_{\R^3}\langle f(x,tE),E\rangle\, dxdt=\int_{\R^3}F(x,E)\, dx.$$
\end{proof}

\begin{altproof}{Lemma \ref{LemConvWeakIpliesStrong}}
	We show that (up to a subsequence) $E_n(x)\to E(x)$ a.e. on $\R^3$.
	Since $I(E_n)\to I(E)$, we have
	\begin{equation} \label{i3}
	\lim_{n\to\infty}\int_{\R^3}F(x,E_n)\,dx=\int_{\R^3}F(x,E)\,dx.
	\end{equation}
	Then from (F2) we infer that for any $0<r\leq R$,
	\begin{equation} \label{mrR}
	m_{r,R}:=\inf_{\substack{x,u_1,u_2\in\R^3\\ r\leq|u_1-u_2|,\\|u_1|,|u_2|\leq R} }\;
	\frac{1}{2}(F(x,u_1)+F(x,u_2))-F\Big(x,\frac{u_1+u_2}{2}\Big)>0.
	\end{equation} 
	Observe that by \eqref{i3} and convexity of $F$,
	$$
	0
	\leq \limsup_{n\to\infty}\int_{\R^3}\frac12(F(x,E_n) + F(x,E))
	- F\left(x,\frac{E_n+E}{2}\right)\,dx
	\leq 0.
	$$
	Therefore, setting
	$$
	\Om_n:=\{x\in\R^3: |E_n-E|\geq r,\;|E_n|\leq R,\;
	|E|\leq R\},
	$$
	there holds
	$$
	|\Om_n| m_{r,R}
	\leq \int_{\R^3}\frac12(F(x,E_n) + F(x,E))
	- F\left(x,\frac{E_n+E}{2}\right)\, dx,
	$$
	and thus $|\Om_n|\to 0$ as $n\to\infty$. Since $0<r\leq R$ are arbitrarily chosen, we deduce
	$$
	E_n\to E\hbox{ a.e.\ on }\R^3.
	$$
	In view of Lemma \ref{LemBrezLieb},  we obtain
	$$\int_{\R^3}F(x,E_n)\, dx-\int_{\R^3}F(x,E_n-E)\, dx\to\int_{\R^3}F(x,E)\, dx$$
	and hence
	$$\int_{\R^3}F(x,E_n-E)\, dx\to0.$$
	By (F3) and Lemma \ref{AllProp} (iii) we get $|E_n-E|_\Phi\to 0$.
\end{altproof}

\begin{Prop}\label{PropDefOfm(u)} Conditions  (I1)--(I8) are satisfied and
 there is a Cerami sequence $(v_n,w_n)\subset \cM$ at the level $c_{\cN}$, i.e.   $\J(v_n,w_n)\to c_{\cN}$ and $(1+\|(v_n,w_n)\|)\J'(v_n,w_n)\to 0$ as $n\to\infty$, where
	$$c_{\cN}:=\inf_{(v,w)\in\mathcal{N}}\J(v,w)>0.$$
\end{Prop}
\begin{proof}
	Setting $X:=\V\times\W$,
	$X^+:=\V\times\{0\}$ and $\tX:=\{0\}\times\V$ we check assumptions (I1)--(I8) for the functional $\J:X\to\R$ 
	given by
	$$\J(v,w)=\frac{1}{2}\|v\|^2_\D-\cI(v,w)$$
	(cf. \eqref{eqJ} and \eqref{DefOfXi}). Recall 
	\[
	\|(v,w)\|:=\bigl(\|v\|_{\cD}^2+|w|_{\Phi}^2\bigr)^{\frac{1}{2}}, \quad \text{where }  \|v\|_{\cD} = |\nabla v|_2.
	\]
	Convexity and differentiability of $\cI$, (F3) and Lemma \ref{LemConvWeakIpliesStrong} yield:
	\begin{itemize}
		\item[(I1)] $\cI|_{\V\times\W}\in \cC^1(\V\times\W,\R)$ and $\cI(v,w)\geq \cI(0,0)=0$ for any $(v,w)\in \V\times\W$.
		\item[(I2)] If $v_n\to v$ in $\V$, $w_n\weakto w$ in $\W$,  then $\displaystyle\liminf_{n\to\infty}\cI(v_n,w_n)\geq \cI(v,w)$.
		\item[(I3)] If $v_n\to v$ in $\V$, $w_n\weakto w$ in $\W$ and $\cI(v_n,w_n)\to \cI(v,w)$, then $(v_n,w_n)\to (u,w)$.
	\end{itemize}
	Moreover,
	\begin{itemize}
		\item[(I6)] There exists $r>0$ such that $\inf_{\|v\|_\D=r}\J(v,0)>0$.
	\end{itemize}
	Indeed, by (F3) and (N2) there exist $C$, $C'>0$ (cf. proof of Lemma \ref{Emb}) such that for any $v\in\V$
	$$\J(v,0)=\|v\|_\D^2 -\int_{\R^3}F(x,v)\,dx
	\geq\|v\|_\D^2 -C\int_{\R^3}|v|^6\,dx
	\geq\|v\|_\D^2 -C'\|v\|_\D^6$$
	and thus (I6) is satisfied.
	It is easy to verify using (F3) and (iv) of Lemma \ref{AllProp} that
	\begin{itemize}
		\item[(I4)] $\|v\|_\D+\cI(v,w)\to\infty$ as $\|(v,w)\|\to\infty$.
	\end{itemize}
	Hence also
	\begin{itemize}
	\item[(I5)] If $(v,w)\in\cM$, then $\cI(v,w)<\cI(v,w+\psi)$ for any $\psi\in \cW\setminus\{0\}$
	\end{itemize}
	holds by strict convexity of $F$.
	Next we  prove 
	\begin{itemize}
		\item[(I7)] $\cI(t_n(v_n,w_n))/t_n^2\to\infty$ if $t_n\to\infty$ and $v_n\to v$ for some $v\neq 0$ as $n\to\infty$.
	\end{itemize}
	Observe that by (F3)
	\begin{eqnarray} \label{eq:I7check}
		\cI(t_n(v_n,w_n))/t_n^2&=&
		\int_{\R^3}F(x,t_n(v_n+w_n))/t_n^2\,dx\\ \nonumber
		&\geq& c_2
		\int_{\R^3}\Phi(t_n|v_n+w_n|)/t_n^2\,dx\\ 
		&=& c_2
		\int_{\R^3}\frac{\Phi(t_n|v_n+w_n|)}{t_n^2|v_n+w_n|^2}|v_n+w_n|^2\,dx. \nonumber
	\end{eqnarray} 
	Take $R_0>0$ such that $v\neq 0$ in $L^2(B(0,R_0))$.
	In view of (N3) we find $C>0$ such that
	$$C\Phi(t)\geq t^2\quad\hbox{for }t\geq 1.$$
	Then 
	\begin{equation} \label{estim}
	\int_{B(0,R)}|v_n+w_n|^2\,dx\leq C\int_{\R^3}\Phi(t_n|v_n+w_n|)/t_n^2\,dx+\int_{B(0,R)\cap\{|v_n+w_n|\leq 1\}}|v_n+w_n|^2\,dx
	\end{equation}
	and $\cI(t_n(v_n,w_n))/t_n^2\to\infty$ provided $v_n+w_n
	$ is unbounded in $L^2(B(0,R),\R^3)$ for some $R\geq R_0$. Now, suppose that $v_n+w_n$ is bounded in $L^2(B(0,R),\R^3)$ for any $R\geq R_0$.  
We may assume passing to a subsequence that $v_n\to v$ a.e. and $w_n\weakto w$ in $L^2_{loc}(\R^3,\R^3)$ for some $w$. Given $\eps>0$, let 
\begin{equation}\label{aetoinfI7}
	\Omega_n := \{x\in\R^3: |v_n(x)+w_n(x)|\ge\eps\}.
	\end{equation}
We claim that there exists $\eps>0$ such that $\lim_{n\to\infty}|\Omega_n|>0$, possibly after passing to a subsequence. 
Arguing indirectly, suppose this limit is 0 for each $\eps$. Then $v_n+w_n\to 0$ in measure, so up to a subsequence $v_n+w_n\to 0$ a.e., hence $w_n\to -v$ a.e. and  $w_n\rh -v$ in $L^2_{loc}(\R^3,\R^3)$. Since $\curlop w_n=0$ in the distributional sense, the same is true of $v$. Thus there is $\xi\in H^1_{loc}(\R^3)$ such that $v=\nabla \xi$, see \cite[Lemma 1.1(i)]{Leinfelder}. As $\div(\nabla \xi)=\div v =0$, it follows that $\xi$, and therefore $v$, is harmonic. Recalling that $v\in\cD$, we obtain $v=0$ as in the proof of Lemma \ref{defof}. This is a contradiction.
Taking $\eps$ in \eqref{aetoinfI7} such that $\lim_{n\to\infty}|\Omega_n|>0$, we obtain
\[
\int_{\R^3}\frac{\Phi(t_n|v_n+w_n|)}{t_n^2|v_n+w_n|^2}|v_n+w_n|^2\,dx \ge \int_{\Om_n}\frac{\Phi(t_n|v_n+w_n|)}{t_n^2|v_n+w_n|^2}|v_n+w_n|^2\,dx \to \infty.
\]
\indent	Finally, Proposition \ref{Propuv_N} shows that
	\begin{itemize}
		\item[(I8)]
		$\frac{t^2-1}{2}\cI'(v,w)[(v,w)]+t\cI'(v,w)[(0,\psi)]+\cI(v,w)-\cI(tv,tw+\psi)\leq 0$
		for any $t\geq 0$, $v\in\V$ and $w,\psi\in\W$.
	\end{itemize}
Applying Theorem \ref{ThLink1} we obtain the last conclusion.
\end{proof}

Since there is no compact embedding of $\V$ into $L^{\Phi}$ we cannot expect that the Palais-Smale or Cerami condition is satisfied. 
We need the following variant of Lions' lemma.

\begin{Lem}\label{lem:Conv}
	Suppose that   $(v_n)\subset \D$ is bounded and for some $r>\sqrt{3}$
	\begin{equation}\label{eq:LionsCond11}
	\sup_{y\in \Z^3}\int_{B(y,r)}|v_n|^2\,dx\to 0 \quad\hbox{as } n\to\infty.
	\end{equation}
	Then  
	$$\int_{\R^3} \Phi(|v_n|)\, dx\to 0\quad\hbox{as } n\to\infty.$$
\end{Lem}
\begin{proof}
This follows from \cite[Lemma 1.5]{MederskiZeroMass} since $\Phi$ satisfies (N2).
\end{proof}

We collect  further properties of $\cI$.
\begin{Lem}\label{LemDefofW}$\hbox{}$\\
	$(a)$ For any $v\in L^{\Phi}$ there is a unique $w(v)\in \cW$ such that
	\begin{equation}\label{DefofW(u)}
	\cI(v,w(v))=\inf_{w\in\W}\cI(v,w).
	\end{equation}
	Moreover, $w:L^{\Phi}\to\W$  is continuous.\\
	$(b)$ $w$ maps bounded sets into bounded sets and $w(0)=0$.
\end{Lem}
\begin{proof}
	$(a)$ Let $v\in L^{\Phi}$. 
	Since  $\W\ni w\mapsto \cI(v,w)\in\R$ is continuous, strictly convex and coercive,  there exists a unique $w(v)\in\W$ such that (\ref{DefofW(u)}) holds.
	We show that the map $w:L^{\Phi}\to\W$ is continuous. Let $v_n\to v$ in $L^{\Phi}$. Since 
	\begin{equation}\label{Lem_m_map_ineq}
	0\leq \cI(v_n,w(v_n))\leq \cI(v_n,0),
	\end{equation}
	$w(v_n)$ is bounded and we may assume  $w(v_n)\rightharpoonup w_0$ for some $w_0\in\W$.
	Observe that by the (sequential) lower semi-continuity of $\cI$ we get
	$$\cI(v,w(v))\leq \cI(v,w_0)\leq \liminf_{n\to\infty} 
	\cI(v_n,w(v_n))\leq 
	\liminf_{n\to\infty} 
	\cI(v_n,w(v))=\cI(v,w(v)).$$
	Hence $w(v)=w_0$ and by Lemma \ref{LemConvWeakIpliesStrong} we have $v_n+w(v_n)\to v+w(v)$ in $L^{\Phi}$. Thus $w(v_n)\to w(v)$ in $\W$.\\ 
	$(b)$ This follows from inequality (\ref{Lem_m_map_ineq}), (F3) and Lemma \ref{AllProp} (iv).
\end{proof}

Let $m(v):=(v,w(v))\in\cM$ for $v\in\cV$. Then in view of Lemma \ref{LemDefofW} (a), $m:\cV\to\cM$ is continuous.
The following lemma implies that any Cerami sequence of $\J$ in $\cM$ and any Cerami sequence of $\cJ\circ m$
are bounded.

\begin{Lem}\label{LemCoercive} If
$(v_n)\subset \cV$ is such that $(\J\circ m)(v_n)\leq\beta$ and $(1+\|v_n\|)(\J\circ m)'(v_n)\to 0$ as $n\to\infty$, then $(v_n)$ is bounded.
\end{Lem}

\begin{proof}
	Suppose that $m(v_n)=(v_n,w_n)\in\mathcal{M}$, $\|(v_n,w_n)\|\to\infty$ as $n\to\infty$ and  $\J(v_n,w_n)\le \beta$. Since $w_n=w(v_n)$, $\|(v_n,w_n)\|\to\infty$ if and only if $\|v_n\|_\D\to\infty$.
Let $\bar{v}_n:=v_n/\|v_n\|_\D$ and $\bar{w}_n:=w_n/\|v_n\|_\D$. 
Assume 
$$\lim_{n\to\infty}\sup_{y\in\Z^3}\int_{B(y,r)}|\bar{v}_n|^2\,dx=0$$
for some fixed $r>\sqrt{3}$. 
By Lemma \ref{lem:Conv}, $\lim_{n\to\infty}\int_{\R^3}\Phi(|\bar{v}_n|)\,dx=0$,  and arguing similarly as Liu \cite{Liu}, we obtain a contradiction. More precisely, recalling $\cJ'(v_n,w_n)[w_n]=0$,
Proposition \ref{Propuv_N} with $t_n=s/\|v_n\|_\D$ and $\psi_n=-t_nw_n$ implies that for every $s>0$,
\begin{eqnarray*}
\beta &\ge&\limsup_{n\to\infty}\J(v_n,w_n) \\ &\ge&\limsup_{n\to\infty}\J(s\bar{v}_n,0)-\lim_{n\to\infty}\J'(v_n,w_n)\Big[\Big(\frac{t_n^2-1}{2}v_n,-\frac{t_n^2+1}{2}w_n\Big)\Big]=\limsup_{n\to\infty}\J(s\bar{v}_n,0)\\
&\overset{(F3)}{\ge}&\frac{s^2}{2}-\lim_{n\to\infty}c_1\int_{\R^3}\Phi(s|\bar v_n|)\,dx
 =\frac{s^2}{2}
\end{eqnarray*}
which is impossible. Hence $\liminf_{n\to\infty}\int_{B(y_n,r)}|\bar{v}_n|^2\,dx>0$ for some sequence $(y_n)\subset\Z^3$.
Since $\cM$ and $\J$ are invariant with respect to $\mathbb{Z}^3$-translations, we may assume that
$$\int_{B(0,r)}|\bar{v}_n|^2\,dx\geq c>0$$
for all $n$ sufficiently large and some constant $c$. This implies that up to a subsequence, $\bar{v}_n\rightharpoonup\bar v\ne 0$ in $\D$, $\bar{v}_n\to\bar{v}$ in $L^2_{loc}(\R^3,\R^3)$ and $\bar{v}_n\to\bar v$ a.e. in $\R^3$ for some $\bar v\in\cD$.
By (F4),
\[
2\J(v_n,w_n) - \J'(v_n,w_n)[(v_n,w_n)] = \int_{\R^3}(\langle f(x,v_n+w_n), v_n+w_n\rangle - 2F(x,v_n+w_n))\,dx \ge 0,
\]
so $\cJ(v_n,w_n)$ is bounded below and
$$\alpha \le \frac{\J(v_n,w_n)}{\|v_n\|_\D^2}\le \frac{1}{2}\|\bar{v}_n\|_\D^2-c_2\int_{\R^3}\frac{\Phi(v_n+w_n)}{|v_n+w_n|^2}|\bar{v}_n+\bar{w}_n|^2\,dx$$
for some constant $\alpha$ (cf. \eqref{eq:I7check} for the second inequality). Hence it suffices to show that the integral on the right-hand side above goes to $+\infty$. We can argue as in the proof of (I7) in Proposition \ref{PropDefOfm(u)}. In particular, \eqref{estim} holds with $v_n+w_n$ replaced by $\bar v_n+\bar w_n$ and $t_n$ replaced by $\|v_n\|_\D$, and if $\Om_n$ is as in \eqref{aetoinfI7} (again, with $v_n+w_n$ replaced by $\bar v_n+\bar w_n$), then $\lim_{n\to\infty}|\Om_n|>0$ for a subsequence.
\end{proof}

\begin{Cor}\label{eq:corCeramibounded}
	Let $\beta>0$. There exists $M_\beta>0$ such that for every $(v_n)\subset\V$ satisfying $0\le\liminf_{n\to\infty}\J\bigl(m(v_n)\bigr)\le\limsup_{n\to\infty}\J\bigl(m(v_n)\bigr)\le\beta$ and $\lim_{n\to\infty}(1+\|v_n\|)\J'\bigl(m(v_n)\bigr)=0$ there holds $\limsup_{n\to\infty}\|v_n\|\le M_\beta$.
\end{Cor}

\begin{proof}
If no finite bound $M_\beta$ exists, for each $k$ there is a sequence $(v_n^k)$ satisfying the assumptions above and such that $\limsup_{n\to\infty}\|v_n^k\|\ge k$. Now it is easy to find $n(k)$ in such a way that $(v_{n(k)}^k)$ is
an unbounded sequence satisfying the hypotheses of Lemma \ref{LemCoercive}, a contradiction.
\end{proof}

\section{Weak-to-weak$^*$ convergence in $\cM$}\label{SectAnalysis}

\begin{Lem}\label{CoEmb} Suppose that $\Omega$ is a bounded Lipschitz domain. Then
$H^1(\Omega)$ is compactly embedded in $L^\Phi(\Omega)$. 
\end{Lem}
\begin{proof}
Suppose $u_n\rh 0$ in $H^1(\Om)$. Then $u_n\rh 0$ in $L^6(\Om)$, $u_n\to 0$ in $L^2(\Om)$ and $u_n\to 0$ a.e. in $\Om$ after passing to a subsequence. By (N2), for each $\eps>0$ there exists $C_\eps$ such that $\Phi(t)\le \eps t^6$ for $t>C_\eps$. Hence
\[
\int_{\Om} \Phi(|u_n|)\,dx = \int_{\Om\cap\{|u_n|\le C_\eps\}} \Phi(|u_n|)\,dx + \int_{\Om\cap\{|u_n| > C_\eps\}} \Phi(|u_n|)\,dx \le \int_{\Om\cap\{|u_n|\le C_\eps\}} \Phi(|u_n|)\,dx +C\eps
\]
where the constant $C$ depends only on the $L^6$- bound on $(u_n)$. By the dominated convergence theorem and since $\eps$ is arbitrary, $\int_{\Om} \Phi(|u_n|)\,dx\to 0$ and $|u_n|_\Phi\to 0$ according to Lemma \ref{AllProp}(iii).
\end{proof}

\begin{Prop} \label{prop}
	If $v_n\rh v$ in $\cD$, then $w(v_n)\rh w(v)$ in $\cW$ and, after passing to a subsequence, $w(v_n)\to w(v)$ a.e. in $\R^3$.
\end{Prop}

\begin{proof}
	It follows from the definition \eqref{DefofW(u)} of $w(v)$ that
	\begin{equation} \label{eq1}
	\int_{\R^3}\langle f\bigl(x,v_n+w(v_n)\bigr),z\rangle\,dx = 0 = \int_{\R^3}\langle f\bigl(x,v+w(v)\bigr),z\rangle\,dx \quad \text{for all } z \in \cW.
	\end{equation}
	Since the sequence $(v_n)$ is bounded, so is $\bigl(w(v_n)\bigr)$ by Lemma \ref{LemDefofW}$(b)$. Hence we may assume $w(v_n)\weakto w_0$ for some $w_0$. In addition, since $v_n\to v$ in $L^2_{loc}(\R^3,\R^3)$, then $v_n\to v$ a.e. after passing to a subsequence.\\
\indent	Let $\Omega\subset\R^3$ be bounded and let $\zeta\in \cC_0^\infty(\R^3, [0,1])$ be such that $\zeta=1$ in $\Omega$. By (F3) and Lemmas \ref{AllProp}(ii), \ref{forC1}, \ref{CoEmb}, for some constant $C>0$ we have
	\begin{equation} \label{eq10}
	0\le\int_{\R^3}|f\bigl(x,v_n+w(v_n)\bigr)|\,|v_n-v|\,\zeta\,dx\le C|\Phi'(|v_n+w(v_n)|)|_{\Psi} |(v_n-v)\zeta|_\Phi\to 0.
	\end{equation}
	Choose $R$ so that $\text{supp\,}\zeta\subset B(0,R)$. By (N3), $\bigl(w(v_n)\bigr)$ is bounded in $L^2(B(0,R),\R^3)$. Indeed,
	\[
	C_1 \ge \int_{B(0,R)\cap\{|w(v_n)|\ge 1\}}\Phi(|w(v_n)|)\,dx \ge C_2  \int_{B(0,R)\cap\{|w(v_n)|\ge 1\}}|w(v_n)|^2\,dx
	\]
for  suitable $C_1,C_2>0$.  By \cite[Lemma 1.1$(i)$]{Leinfelder}, for every $n$  there exists $\xi_n\in H^1(B(0,R))$ such that $w(v_n)=\nabla\xi_n$. We may assume $\int_{B(0,R)}\xi_n\,dx = 0$. Then by the Poincar\'e inequality,
	\[
	\|\xi_n\|_{H^1(B(0,R))}\le C'|\nabla\xi_n|_{L^2(B(0,R))}\le C''
	\]
for some $C'$, $C''>0$. Hence in view of Lemma \ref{CoEmb}, up to a subsequence, $\xi_n\to \xi$ in  $L^\Phi\bigl(B(0,R)\bigr)$  for some $\xi\in H^1(B(0,R))$. Similarly as in \eqref{eq10}, we have
	\begin{equation} \label{eq11}
	\lim_{n\to\infty}\int_{\R^3}|f\bigl(x,v_n+w(v_n)\bigr)|\,|\nabla\zeta|\,|\xi_n-\xi|\,dx=0.
	\end{equation}
The limits in \eqref{eq10} and \eqref{eq11} are 0 also if $f(x,v_n+w(v_n))$ is replaced by $f(x,v+\nabla\xi)$. Combining \eqref{eq1}-\eqref{eq11} we obtain
	\begin{equation} \label{eq2}
	\lim_{n\to\infty} \int_{\R^3}\langle f\bigl(x,v_n+w(v_n)\bigr)-f\bigl(x,v+\nabla\xi\bigr),\zeta \bigl(v_n-v+w(v_n)-\nabla\xi\bigr)\rangle\,dx = 0 
	\end{equation}
	where we have taken $z = \nabla\bigl(\zeta (\xi_n-\xi)\bigr)$ in \eqref{eq1}. We shall show that  $v_n+w(v_n) \to v+ \nabla\xi$ a.e. in $\Omega$. The convexity of $F$ in $u$ implies that 
	$$F\Big(x,\frac{u_1+u_2}{2}\Big)\geq F(x,u_1)+\Big\langle f(x,u_1),\frac{u_2-u_1}{2}\Big\rangle$$
	and
	$$F\Big(x,\frac{u_1+u_2}{2}\Big)\geq F(x,u_2)+\Big\langle f(x,u_2),\frac{u_1-u_2}{2}\Big\rangle.$$
	Adding these inequalities and using (F2), we obtain for any $0<r\leq R$ and $|u_1-u_2|\ge r$, $|u_1|, |u_2| \le R$ that
	\[
	m_{r,R} \le \frac12(F(x,u_1)+F(x,u_2)) - F\left(x,\frac{u_1+u_2}2\right)\le \frac14\langle f(x,u_1)-f(x,u_2), u_1-u_2\rangle
	\]
	 where $m_{r,R}$ has been defined in \eqref{mrR}. Since $\zeta=1$ in $\Om$, it is now easy to see from \eqref{eq2} that $v_n+w(v_n) \to v+ \nabla\xi$ a.e. in $\Omega$ as claimed. Since $w(v_n)\weakto w_0$, $w_0=\nabla\xi$ and by the usual diagonal procedure we obtain a.e. convergence to $v+w_0$ in $\R^3$. 
	 Take any $w\in\cW$ and observe that by the Vitali convergence theorem 
	 $$0=\int_{\R^3}\langle f(x,v_n+w(v_n)), w\rangle\,dx\to \int_{\R^3}\langle f(x,v+w_0), w\rangle\,dx.$$
	 The uniqueness of a minimizer (see Lemma \ref{LemDefofW}) implies that $w_0=w(v)$.
	 \\
	 \indent So far we have shown that if $v_n\rh v$ in $\cD$, then a subsequence of $(w(v_n))$ converges a.e. in $\R^3$, and therefore weakly in $\W$, to $w(v)$. But since each subsequence of $(w(v_n))$ has a subsequence converging weakly to $w(v)$, we can conclude that $w(v_n)\rh w(v)$ for the full sequence. 
\end{proof}

In general $\J'$ is not (sequentially) weak-to-weak$^*$ continuous,
however we show the weak-to-weak$^*$ continuity of $\J'$ for sequences on the topological manifold $\mathcal{M}$. Obviously, the same regularity holds for $\cE'$ and $\cM_{\cE}$.

\begin{Cor}\label{CorJweaklycont}
	If $(v_n,w_n)\in\mathcal{M}$ and $(v_n,w_n)\rightharpoonup (v_0,w_0)$ in $\V\times\W$ then $\J'(v_n,w_n)\rightharpoonup \J'(v_0,w_0)$, i.e.
	$$\J'(v_n,w_n)[(\phi,\psi)]\to \J'(v_0,w_0)[(\phi,\psi)]$$
	for any $(\phi,\psi)\in\V\times\W$.
\end{Cor}

\begin{proof}
	By Lemma \ref{LemDefofW}$(a)$ we get $w_n=w(v_n)$. In view of Proposition \ref{prop}, we may assume  $v_n+w_n\to v_0+w_0$ a.e. in $\R^3$ (where $w_0=w(v_0)$).
	For $(\phi,\psi)\in\V\times\W$ we have
	\begin{eqnarray*}
		\J'(v_n,w_n)[(\phi,\psi)]-\J'(v_0,w_0)[(\phi,\psi)]&=&\int_{\R^3}\langle \nabla v_n-\nabla v_0,\nabla \phi\rangle\,dx\\
		&&-\int_{\R^3}\langle f(x,v_n+w_n)-f(x,u_0+w_0),\phi+\psi\rangle\,dx. 
	\end{eqnarray*}
We may assume $\phi,\psi$ are compactly supported. Let $\Om$ be a bounded set containing the support of $\phi+\psi$. Then
\[
\int_{\Om} |\langle f(x,v_n+w_n)-f(x,u_0+w_0),\phi+\psi\rangle|\,dx \le |f(x,v_n+w_n)-f(x,u_0+w_0)|_{L^\psi(\Om)} |\phi+\psi|_{L^\Phi(\Om)}
\]
(cf. \eqref{*}).	In view of the Vitali convergence theorem and uniform integrability of the norm \cite[Theorem III.4.14]{RaoRen}, we obtain
	$$\J'(v_n,w_n)[(\phi,\psi)]-\J'(v_0,w_0)[(\phi,\psi)]\to 0.$$
\end{proof}

\section{Proof of Theorem \ref{ThMain}}

Recall that the group $G:=\Z^3$ acts isometrically by translations on $X=\cV\times\cW$ and $\cJ$ is $\Z^3$-invariant.
Let
$$\cK:=\big\{v\in \cV: (\cJ\circ m)'(u)=0\big\}$$
and suppose that $\cK$ consists of a finite number of distinct orbits. It is clear that $\Z^3$ acts discretely and hence satisfies the condition (G) in Section \ref{sec:criticaslpoitth}.
Then, in view of Lemma \ref{discrete},
$$\kappa:= \inf\big\{\|v-v'\|_{\D}:\J'\bigl(m(v)\bigr) = \J'\bigl(m(v')\bigr) = 0, v\ne v'\big\}>0.$$

\begin{Lem}\label{Discreteness}
Let $\beta\ge c_{\cN}$ and suppose that $\cK$ has a finite number of distinct orbits.  
If $(u_n),(v_n)\subset\V$ are two Cerami sequences for $\J\circ m$ such that $0\le\liminf_{n\to\infty}\J\bigl(m(u_n)\bigr)\le \limsup_{n\to\infty}\J\bigl(m(u_n)\bigr)\le\beta$, $0\le\liminf_{n\to\infty}\J\bigl(m(v_n)\bigr)\le \limsup_{n\to\infty}\J\bigl(m(v_n)\bigr)\le\beta$ and $\liminf_{n\to\infty}\|u_n-v_n\|_{\D}< \kappa$, then $\lim_{n\to\infty}\|u_n-v_n\|_{\D}=0$.
\end{Lem}

\begin{proof}
Let $m(u_n)=(u_n,w^1_n)$, $m(v_n)=(v_n,w^2_n)$. By Corollary \ref{eq:corCeramibounded}, $m(u_n)$, $m(v_n)$ are bounded.
We first consider the case 
\begin{equation}
\label{pqconvergence}
\lim_{n\to\infty}|u_n-v_n|_{\Phi}=0
\end{equation}
and  prove that
\begin{equation}
\label{convergence}
\lim_{n\to\infty}\|u_n-v_n\|_\D=0.
\end{equation}
By (F3) and Lemmas \ref{forC1}, \ref{LemCoercive}, we have
\[
\begin{split}
\|u_n-v_n\|_{\mathcal{D}}^2 = \, & \J'(m(u_n))[(u_n-v_n,0)]-\J'(m(v_n))[(u_n-v_n,0)]\\
& +\int_{\R}\langle f(x,m(u_n))-f(x,m(v_n)),u_n-v_n\rangle\,dx\le\\
\le \, & o(1)+\int_{\R}\bigl(|f(x,m(u_n))|+|f(x,m(v_n))|\bigr)|u_n-v_n|\,dx\\
\le \, & o(1)+c_1\int_{\R}\big(\Phi'(|m(u_n)|)+\Phi'(|m(v_n)|)\big)|u_n-v_n|\,dx \\
\le \, & o(1)+ c_2\left(|\Phi'(|m(u_n)|)|_\Psi+|\Phi'(|m(v_n)|)|_\Psi\right)|u_n-v_n|_{\Phi} \ \to\ 0
\end{split}
\]
which gives \eqref{convergence}.\\
\indent 
Suppose now \eqref{pqconvergence} does not hold. By Lemma \ref{AllProp} (iii) and Lemma \ref{lem:Conv}, for a fixed $R>\sqrt{3}$ there exist $\eps>0$ and a sequence $(y_n)\subset\Z^3$ such that, passing to a subsequence,
\begin{equation}
\label{boundedaway}
\int_{B(y_n,R)}|u_n-v_n|^2\,dx\ge\eps.
\end{equation}
Since $\J$ is $\Z^3$-invariant, we may assume $y_n=0$. As $m(u_n), m(v_n)$ are bounded, up to a subsequence,
\begin{equation}
\label{weakconvergence}
(u_n,w^1_n)\rightharpoonup(u,w^1)\hbox{ and }(v_n,w^2_n)\rightharpoonup(v,w^2)\quad\hbox{in }\V\times\W 
\end{equation}
for some $(u,w^1),(v,w^2)\in \cV\times\cW$. As $u_n\to u$ and $v_n\to v$ in $L^2_{loc}(\R^3,\R^3)$, $u\ne v$ according to \eqref{boundedaway}.
From Corollary \ref{CorJweaklycont} and \eqref{weakconvergence} we infer that
\[
\J'(u,w^1)=\J'(v,w^2)=0.
\]
Thus
\[
\liminf_{n\to\infty}\|u_n - v_n\|_{\D} \ge\|u - v\|_{\D}\ge \kappa
\]
which is a contradiction.
\end{proof}

\begin{altproof}{Theorem \ref{ThMain}}\mbox{}\\ 
(a) The existence of a Cerami sequence $((v_n,w_n))\subset \cM$ at the level $c_{\cN}$ follows from Proposition \ref{PropDefOfm(u)}, and this sequence is bounded by 
Corollary \ref{eq:corCeramibounded}. Similarly as in the proof of Lemma \ref{Discreteness} we find  $v\in\cV\setminus\{0\}$ such that
$(v_n, w_n) \weakto (v,w)$ and $(v_n, w_n) \to (v,w)$  a.e. in $\R^3$ along a subsequence and $\cJ'(v,w)=0$ (with $w=w(v)$). More precisely, if $|v_n|_\Phi\to 0$, then \eqref{convergence} with $u_n=0$ holds by the same argument. This is impossible because $\J'(m(v_n))\to c_\cN>0$. Hence \eqref{boundedaway} with $u_n=0$ is satisfied and we may assume making translations by $y_n$ if necessary that $\int_{B(0,R)}|v_n|^2\,dx\ge \eps$. So $v\ne 0$.
By Fatou's lemma and (F4),
\begin{eqnarray*}
c_{\cN}&=&\lim_{n\to\infty}\cJ(v_n,w_n)=\lim_{n\to\infty}\Big(\cJ(v_n,w_n)-\frac12\cJ'(v_n,w_n)[(v_n,w_n)]\Big) \\
&\geq&\cJ(v,w)-\frac12\cJ'(v,w)[(v,w)] =\cJ(v,w).
\end{eqnarray*}
Since $(v,w)\in\cN$,  $\J(v,w)=c_{\cN}$ and $E=v+w$ solves \eqref{eq}. Note that here we have \emph{not} assumed $\cK$ has finitely many distinct orbits. \\
(b) In order to complete the proof we use directly Theorem \ref{Th:CrticMulti}(b). That (I1)--(I8) are satisfied and $(M)_0^\beta$ holds for all $\beta>0$ follow from Proposition \ref{PropDefOfm(u)}, Corollary \ref{eq:corCeramibounded} and Lemma \ref{Discreteness}.
\end{altproof}

\section{A remark on the Schr\"odinger equation}\label{sec:Schordingerproblem}

Theorem \ref{Th:CrticMulti} can also be used to deal with the Schr\"odinger equation or a system of equations. In particular, one can use it to obtain alternative proofs of the results in \cite{dePaivaKryszewskiSzulkin, SzulkinWeth}. Contrary to \cite{dePaivaKryszewskiSzulkin}, we do not need to use nonsmooth critical point theory.\\
\indent  Below we  briefly discuss a very simple application of Theorem \ref{Th:CrticMulti}, yet our result extends and complements known ones. We leave the details to the reader.
We look for solutions to the equation
\begin{equation}\label{eqS}
-\Delta u= f(x,u), \quad  x\in\R^N, \ N\ge 3.
\end{equation}
The functional
\begin{equation*}
\J(u):=\frac12\int_{\R^N}|\nabla u|^2\,dx - \int_{\R^N} F(x,u)\,dx
\end{equation*}
corresponding to \eqref{eqS} is of class $\cC^1$ on  $\D^{1,2}(\R^N)$ if $f$ satisfies the following assumptions:
\begin{itemize}
	\item[(AF1)] $F:\R^N\times\R\to\R$ is differentiable with respect to the second variable $u\in\R$ and $f:=\pa_uF:\R^N\times\R\to\R$ is a Carath\'eodory function (i.e.\ measurable in $x\in\R^N$, continuous in $u\in\R$ for a.e.\ $x\in\R^N$). Moreover, $f$ is $\Z^N$-periodic in $x$, i.e. $f(x,u)=f(x+y,u)$ for $x\in\R^N,u\in \R$ and $y\in\Z^N$.
	\item[(AF2)] 
	$\displaystyle \lim_{u\to 0}f(x,u)/|u|^{2^*-1}=\lim_{|u|\to\infty}f(x,u)/|u|^{2^*-1}=0$ uniformly in $x$ where $2^* := 2N/(N-2)$.
	\item[(AF3)] $F(x,u)/u^2\to\infty$ uniformly in $x$ as $|u|\to\infty$.
	\item[(AF4)] $f(x,u)/|u|$ is non-decreasing on $(-\infty,0)$ and on $(0,\infty)$.
\end{itemize}
Note that there is no convexity-type assumption similar to (F2). However, (AF4) implies (not necessarily uniform) convexity of $F$ as well as (F4). Since the quadratic part of $\cJ$ is positive definite, we have $X^+=\cM=\cD^{1,2}(\R^N)$ and $\tX=\{0\}$, so $m(u)=u$ here and we easily check (I1)--(I8) from Section \ref{sec:criticaslpoitth}. In fact (I2)--(I4) are trivially satisfied, (I5) is an empty condition and (I8) becomes much simpler because $v$ is necessarily 0. Using Theorems \ref{ThLink1} and \ref{Th:CrticMulti} we obtain the following result. 

\begin{Th}\label{ThMainAppendix} Assume that (AF1)--(AF4) hold. Then: \\
	(a) Equation \eqref{eqS} has a ground state solution, i.e. there is a critical point $u\in\cN$ of $\J$ such that
	$$\J(u)=\inf_{\mathcal{N}}\J>0$$
	where
	\begin{eqnarray*}
	\mathcal{N} &:=& \big\{u\in \D^{1,2}(\R^N): u\neq 0,\; 
	\J'(u)(u)=0\big\}.
	\end{eqnarray*}
	(b) If in addition $F$ is even in $u$, then there is an infinite sequence $(u_n)\subset\cN$ of geometrically distinct solutions of  \eqref{eq}, i.e. solutions such that  $(\Z^N\ast u_n)\cap (\Z^N\ast u_m)=\emptyset$ for $n\neq m$ where
	$$\Z^N\ast u_n:=\{u_n(\cdot+y): y\in\Z^N\}.$$
\end{Th}

Problem \eqref{eqS} with growth of the form (AF2) is the so called {\em zero mass case} introduced in \cite{BerLions} for the  autonomous nonlinearity $f(x,u)=f(u)$. In the nonautonomous case it has been  studied e.g. in \cite{BenGrisantiMich,ClappMaia}, see also the references therein. In \cite{BenGrisantiMich,ClappMaia} more restrictive growth conditions have been imposed. In particular, $F$ is of order $|u|^q$ for small $|u|>0$ and of order $|u|^p$ for $|u|$ large where $2<p<2^*<q$. This makes it necessary to work in the Orlicz space $L^p(\R^N)+L^q(\R^N)$.
 In Theorem \ref{ThMainAppendix} we are able to deal with a class of nonlinearities with less restrictive growth conditions (AF2) and we no longer need to use any Orlicz setting.

{\bf Acknowledgements.}
J.M.  would like to thank the members of the CRC 1173 as well as the members of the Institute of Analysis at Karlsruhe Institute of Technology (KIT), 
where part of this work has been done, for their invitation, support and warm hospitality and he was partially supported by the Deutsche Forschungs\-gemeinschaft (DFG) through CRC 1173. J.M. and J.S. were also supported by the National Science Centre, Poland (Grant No. 2017/26/E/ST1/00817).

{\bf Compliance with Ethical Standards.} The authors declare that they have no conflict of
interests, they also confirm that the manuscript complies to the Ethical Rules applicable for this
journal.

\end{document}